\documentclass[12pt,reqno]{amsart}

\usepackage[T1]{fontenc}
\usepackage{enumerate}
\usepackage{hyperref}
\hypersetup{
  colorlinks   = true,
  citecolor    = blue,
  linkcolor    = blue
}
\usepackage{amsmath,amsfonts,amsthm,amssymb,color,tikz}

\usepackage{pdfsync}
\usepackage[font={scriptsize}]{caption}
\usepackage{ulem}

\usepackage[left=1in, right=1in, top=1.1in,bottom=1.1in]{geometry}
\newcommand{\der}{\delta}

\newcommand{\hep}{\hat{\varepsilon}}

\newcommand{\R}{\mathbb R}
\newcommand{\N}{\mathbb N}

\newcommand{\cac}{\mathcal C}

\newcommand{\cn}{\mathcal N}

\newcommand{\al}{\alpha}

\newcommand{\ep}{\varepsilon}

\newcommand{\ga}{\gamma}

\newcommand{\ka}{\kappa}
\newcommand{\la}{\lambda}

\newcommand{\si}{\sigma}

\newcommand{\lp}{\left(}
\newcommand{\rp}{\right)}
\newcommand{\lc}{\left[}
\newcommand{\rc}{\right]}
\newcommand{\lcl}{\left\{}
\newcommand{\rcl}{\right\}}
\newcommand{\lln}{\left|}
\newcommand{\rrn}{\right|}

\newtheorem{theorem}{Theorem}[section]

\newtheorem{corollary}[theorem]{Corollary}

\newtheorem{definition}[theorem]{Definition}

\newtheorem{hypothesis}[theorem]{Hypothesis}
\newtheorem{lemma}[theorem]{Lemma}

\newtheorem{proposition}[theorem]{Proposition}

\theoremstyle{remark}
\newtheorem{remark}[theorem]{Remark}

\theoremstyle{remark}

\newcommand{\bean}{\begin{eqnarray*}}
\newcommand{\eean}{\end{eqnarray*}}
\newcommand{\ben}{\begin{enumerate}}
\newcommand{\een}{\end{enumerate}}
\newcommand{\beq}{\begin{equation}}
\newcommand{\eeq}{\end{equation}}

\begin{document}

\title[RDEs with power nonlinearities]{Rough differential equations\\ with power type nonlinearities}

\date{\today}

\author[P. Chakraborty \and S. Tindel]{Prakash Chakraborty \and Samy Tindel}

\address{Samy Tindel: Department of Mathematics,
  Purdue University,
  150 N. University Street,
  W. Lafayette, IN 47907,
  USA.}
\email{stindel@purdue.edu}

\address{Prakash Chakraborty: Department of Statistics,
  Purdue University,
  150 N. University Street,
  W. Lafayette, IN 47907,
  USA.}
\email{chakra15@purdue.edu}

\thanks{S. Tindel is supported by the NSF grant  DMS-1613163}

\begin{abstract}
In this note we consider differential equations driven by a signal $x$ which is $\ga$-H\"older with $\ga>\frac13$, and is assumed to possess a lift as a rough path. Our main point is to obtain existence of solutions when the coefficients of the equation behave like power functions of the form $|\xi|^{\ka}$ with $\ka\in(0,1)$. Two different methods are used in order to construct solutions: (i) In a 1-d setting, we resort to a rough version of Lamperti's transform. (ii) For multidimensional situations, we quantify some improved regularity estimates when the solution approaches the origin.
\end{abstract}

\maketitle

\date{\today}

\section{Introduction} \label{sec:intro}

This article is concerned with the following $\R^{m}$-valued integral equation:
\begin{equation}\label{eq:power-SDE}
y_t = a + \sum_{j=1}^{d} \int_{0}^{t} \si^{j} (y_s) dx_{s}^{j}, \quad t\in [0,T] 
\end{equation}
where $x:[0,T] \to \R^d$ is a noisy function in the H\"older space ${\cac}^{\ga}([0,T]; \R^d)$ with $\ga>\frac{1}{3}$, $a\in\R^m$ is the initial value and $\si^j$ are vector fields on $\R^m$. We shall resort to rough path techniques in order to make sense of the noisy integral in equation \eqref{eq:power-SDE}, and we refer to \cite{FV-bk,Gubinelli} for further details on the rough paths theory. Our main goal is to understand how to define solutions to \eqref{eq:power-SDE} when the coefficients $\si^{j}$ behave like power functions.

Indeed, the rough path theory allows to consider very general noisy signals $x$ as drivers of equation \eqref{eq:power-SDE}, but it requires heavy regularity assumptions on the coefficients $\si^{j}$ in order to get existence and uniqueness of solutions. More specifically, given the regularity of the coefficient $\si$, a minimal sufficient regularity of the driving signal that guarantees existence and uniqueness of the solution is provided in \cite{FV-bk}. However, for differential equations driven by Brownian motion (which means in particular that $x \in \cac^{\frac{1}{2}-}$) the condition amounts to the coefficient being twice differentiable. This is obviously far from being optimal with respect to the classical stochastic calculus approach for Brownian motion.

One of the current challenges in rough paths analysis is thus to improve the regularity conditions on  the coefficients of \eqref{eq:power-SDE}, and still get solutions to the differential system at stake. Among the irregular coefficients which can be thought of, power type functions of the form $\si^{j}(\xi)=|\xi|^{\ka}$ with $\ka\in(0,1)$ play a special role. On the one hand these coefficients are related to classical population dynamics models (see e.g \cite{DP} for a review), which make them interesting in their own right. On the other hand, the fact that these coefficients vanish at the origin grant them some special properties which can be exploited in order to construct H\"older-continuous solutions. Roughly speaking, equation \eqref{eq:power-SDE} behaves like a noiseless equation when $y$ approaches 0, and one expects existence of a $\ga$-H\"older solution whenever $\ga+\ka>1$. This heuristic argument is explained at length in the introduction of~\cite{YSDE}, and the current contribution can be seen as the first implementation of such an idea in a genuinely rough context.

Let us now recall some of the results obtained for equations driven by a Brownian motion $B$. For power type coefficients, most of the results concern one dimensional cases of the form:
\begin{equation}\label{eq:power-SDE-brownian}
y_t = a +  \int_{0}^{t} \si (y_s) dB_{s}, \quad t\in [0,T] .
\end{equation}
The classical result \cite[Theorem 2]{WY-1} involves stochastic integrals in the It\^o sense, and gives existence and uniqueness for $\si(\xi)=|\xi|^{\ka}$ with $\ka\ge\frac12$. However, the rough path setting is more related to Stratonovich type integrals in the Brownian case. We thus refer the interested reader to the comprehensive study performed in \cite{Cherny-Engelbert}, which studies singular stochastic differential equations and classifies them according to the nature of their solution. Comparing equation \eqref{eq:power-SDE-brownian} interpreted in the Stratonovich sense with the systems analyzed in \cite{Cherny-Engelbert}, their results can be read as follows: if $\si(\xi)=|\xi|^{\ka}$ with $\ka\ge\frac12$ and the solution of \eqref{eq:power-SDE-brownian} starts at a non-negative location, then it reaches zero almost surely. In addition, among solutions with no sojourn time at zero (i.e their local time at 0 vanishes), there is a non-negative solution which is unique in law. However, in general we do not have uniqueness. The results we will obtain for a general rough path are not as sharp, but are at least compatible with the Brownian case. Let us also mention the works \cite{Myt-1, Myt-2}, where the authors study existence and uniqueness of solutions in the context of stochastic heat equations with space time white noise and power type coefficients.

As far as power type equations driven by general noisy signals $x$ are concerned, we are only aware of the article \cite{YSDE} exploring equation~\eqref{eq:power-SDE} in the Young case $\ga > 1/2$. The current contribution has thus to be seen as a generalization of \cite{YSDE}, allowing to cope with $\ga$-H\"older signals $x$ with $\ga \in (1/3,1/2]$.
As we will see, it turns out that when $\ka+\ga > 1$ equation~\eqref{eq:power-SDE} is well defined and yields  a solution. More specifically, we shall obtain the following theorem in the 1-dimensional case (see Theorem \ref{thm:one_d} for a more precise and general formulation). 

\begin{theorem}\label{thm:1d-power}
Consider a 1-dimensional  signal $x\in\cac^{\ga}$, with $\ga \in (1/3,1/2]$.
Let $\sigma$ be the power function given by $\sigma(\xi) = |\xi|^\ka$ and $\phi$ be the function defined by $\phi(\xi) = \int_0^{\xi} \frac{ds}{\si(s)}$. Assume $\ga \in \left(\frac{1}{3},\frac{1}{2}\right]$ and $\ka + \ga > 1$. Then the function $y=\phi^{-1}(x+\phi(a))$ is a solution of the equation
\begin{equation*}
y_t = a + \int_0^t \si(y_s) dx_s, \quad t \geq 0 .
\end{equation*}
\end{theorem}

In the multidimensional case under a slightly increased regularity assumption on $x$, namely $x \in \cac^{\ga+}([0,T])$ as well as a roughness assumption (see Hypothesis \ref{hyp:roughness} for precise statement), the following theorem holds under a few power type hypotheses on $\si$ and its derivatives.
\begin{theorem}\label{thm:d-dim-power}
Consider a $d$-dimensional  signal $x\in\cac^{\ga}$ with $\ga \in (1/3,1/2]$, giving raise to a rough path. Assume $\ka + \ga > 1$, and that $\sigma(\xi)$ behaves like a power coefficient $|\xi|^\ka$ near the origin. Then there exists a continuous function $y$ defined on $[0,T]$ and an instant $\tau \le T$, such that one of the following two possibilities holds:  

\begin{itemize}
\item[(A)] $\tau =T$: $y$ is non-zero on $[0,T]$,  $y\in\cac^{\ga}([0,T];\R^m)$   and $y$ solves equation  \eqref{eq:sde-power} on $[0,T]$.

\item[(B)] $\tau<T$: the path $y$ sits in $\cac^{\ga}([0,T];\R^m)$   and $y$ solves equation  \eqref{eq:sde-power} on $[0,T]$.
Furthermore,  $y_s \not=0$ on   $[0,\tau)$, $\lim_{t\rightarrow \tau} y_t=0$ and $y_t=0$ on the interval $[\tau,T]$.
\end{itemize}
\end{theorem}

As mentioned above, Theorems \ref{thm:1d-power} and \ref{thm:d-dim-power} are the first existence results for power type coefficients in a truly rough context. As in \cite{YSDE}, their proofs mainly hinge on a quantification of the regularity gain of the solution $y$ when it approaches the origin. We should mention however that this quantification requires a significant amount of effort in the rough case. Indeed we resort to some discrete type expansions, whose analysis is based on precise estimates inspired by the numerical  analysis of rough differential equations (see e.g. \cite{Liu-T}).

Having stated the key results, we now describe the outline of this article. In Section \ref{sec:rough}, a short account of the necessary notions of rough path theory is provided. Section \ref{subsec:setting} deals with a few hypotheses we assume on the coefficient $\sigma$, all of which are satisfied by the power type coefficient $|\xi|^{\ka}$. Section \ref{sec:1d} proves the existence of a solution in the one-dimensional case. In Section \ref{sec:multi-dim} we proceed by considering a few stopping times and quantify the regularity gain mentioned above of the solution when it hits $0$. We achieve this through discretization techniques as employed in Theorem~\ref{thm:R_bnd}. Finally we show H\"older continuity of our solution.

\smallskip

\noindent
\textbf{Notations.} The following notations are used in this article: 
\begin{enumerate}
\item For an arbitrary real $T > 0$, let $\mathcal{S}_{k}([0,T])$ be the $k$th order simplex defined by $\mathcal{S}_{k}([0,T]) = \lbrace (s_1,\ldots,s_k):0 \leq s_1 \leq \cdots \leq s_k \leq T \rbrace$.  
\item For quantities $a$ and $b$, let $a \lesssim b$ denote the existence of a constant $c$ such that $a \leq cb$.
\item For an element $z$ in the functional space $\mathcal{R}$, let $\cn[z;\mathcal{R}]$ denote the corresponding norm of $z$ in $\mathcal{R}$.
\end{enumerate}

\section{Rough Path Notions}\label{sec:rough}
The following is a short account of the rough path notions used in this article, mostly taken from \cite{Gubinelli}. We review the notion of controlled process as well as their integrals with respect to a rough path. We shall also give a version of an It\^o-Stratonovich change of variable formula under reduced regularity condition.

\subsection{Increments}\label{sec:increm}
For a vector space $V$ and an integer $k \geq 1$, let $\cac_k(V)$ be the set of functions $g:\mathcal{S}_{k}([0,T]) \to V$ such that $g_{t_1 \cdots t_k} = 0$ whenever $t_i = t_{i+1}$ for some $i \leq {k-1}$. Such a function will be called a $(k-1)$-increment, and we set $\cac_{*}(V) = \cup_{k \geq 1} \cac_{k}(V)$. Then the operator $\der : \cac_{k}(V) \to \cac_{k+1}(V)$ is defined as follows
\begin{equation}\label{eq:def-delta}
{\der g}_{t_1 \cdots t_{k+1}} = \sum_{i = 1}^{k+1} (-1)^{k-i} g_{t_1\cdots \hat{t_i} \cdots t_{k+1}}
\end{equation} 
where $\hat{t_{i}}$ means that this particular argument is omitted. It is easily verified that $\der \der = 0$ when considered as an operator from $\cac_{k}(V)$ to $\cac_{k+2}(V)$.

\noindent
The size of these $k$-increments are measured by H\"older norms defined in the following way: for $f \in \cac_{2}(V)$ and $\mu > 0$ let
\begin{equation}\label{eq:norm_1}
\|f\|_{\mu} = \sup_{(s,t) \in \mathcal{S}_2([0,T])} \dfrac{\|f_{st}\|}{{|t-s|}^{\mu}}~~~\text{  and  }~~~\cac_2^{\mu}(V) = \lbrace f \in \cac_2(V); {\|f\|}_{\mu} < \infty \rbrace 
\end{equation}
The usual H\"older space $\cac_1^{\mu}(V)$ will be determined in the following way: for a continuous function $g \in \cac_{1}(V)$, we simply set
\begin{equation*}
{\|g\|}_{\mu} = {\|\der g\|}_{\mu}
\end{equation*}
and we will say that $g \in \cac_1^{\mu}(V)$ iff ${\|g\|}_{\mu}$ is finite. 
\begin{remark} 
Notice that ${\|\cdot\|}_{\mu}$ is only a semi-norm on $\cac_1(V)$, but we will generally work on spaces for which the initial value of the function is fixed.
\end{remark}

We shall also need to measure the regularity of increments in $\cac_3(V)$. To this aim, similarly to \eqref{eq:norm_1}, we introduce the following norm for $h \in \cac_3(V)$:
\begin{equation}\label{eq:norm_2}
{\|h\|}_{\mu} = \sup_{(s,u,t) \in \mathcal{S}_3([0,T])} \dfrac{|h_{sut}|}{{|t-s|}^{\mu}}.
\end{equation}
Then the $\mu$-H\"older continuous increments in $\cac_3(V)$ are defined as:
\begin{equation*}
\cac_3^{\mu}(V) := \lbrace h \in \cac_{3}(V); {\|h\|}_{\mu} < \infty \rbrace .
\end{equation*}

The building block of the rough paths theory is the so-called sewing map lemma. We recall this fundamental result here for further use.
\begin{proposition}
Let $h \in \cac_3^{\mu}(V)$ for $\mu > 1$ be such that $\der h = 0$. Then there exists a unique $g = \Lambda(h) \in \cac_2^{\mu}(V)$ such that $\der g = h$. Furthermore for such an $h$, the following relations hold true:
\begin{equation*}
\der \Lambda(h) = h~\text{ and }~{\|\Lambda h\|}_{\mu} \leq \dfrac{1}{2^{\mu} - 2} {\|h\|}_{\mu} .
\end{equation*}
\end{proposition}

\subsection{Elementary computations in \texorpdfstring{$\mathcal{C}_{2}$ and $\mathcal{C}_{3}$}{$\mathcal{C}_{2}$ and $\mathcal{C}_{3}$}}
Consider $V=\R$, and let $\mathcal{C}_{k}^{\ga}$ for $\mathcal{C}_{k}^{\ga}(\R)$. Then $\left(\mathcal{C}_{\ast}, \der\right)$ can be endowed with the following product: for $g \in \mathcal{C}_n$ and $h \in \mathcal{C}_m$ we let $gh$ be the element of $\mathcal{C}_{m+n-1}$ defined by
\begin{equation*}
(gh)_{t_1, \ldots, t_{m+n-1}} = g_{t_1, \cdots, t_n}h_{t_n,\cdots t_{m+n-1}},~~~(t_1, \ldots, t_{m+n-1}) \in \mathcal{S}_{m+n-1}([0,T]).
\end{equation*}
We now label a rule for discrete differentiation of products for further use throughout the article. Its proof is an elementary application of the definition \eqref{eq:def-delta}, ans is omitted for sake of conciseness.
\begin{proposition}\label{prop:der_rules}
The following rule holds true:
Let $g \in \cac_1$ and $h \in \cac_2$. Then $gh \in \cac_2$ and
\begin{equation*}
\der(gh) = \der g\, h - g\, \der h .
\end{equation*}
\end{proposition}

\noindent
The iterated integrals of smooth functions on $[0,T]$ are particular cases of elements of $\cac_2$, which will be of interest. Specifically, for smooth real-valued functions $f$ and $g$, let us denote $\int f dg$ by $\mathcal{I}(f dg)$ and consider it as an element of $\cac_2$: for $(s,t) \in \mathcal{S}_{2}\left([0,T]\right)$ we set
\begin{equation*}
\mathcal{I}_{st} (f dg) = \left(\int f dg\right)_{st} = \int_{s}^{t} f_u dg_u .
\end{equation*}

\subsection{Weakly controlled processes}

One of our basic assumptions on the driving process $x$ of equation \eqref{eq:power-SDE} is that it gives raise to a geometric rough path. This assumption can be summarized as follows.
\begin{hypothesis}\label{hyp:x}
The path $x:[0,T] \to \R^d$ belongs to the H\"older space ${\cac}^{\ga}([0,T];\R^d)$ with $\ga \in \left(\frac{1}{3}, \frac{1}{2}\right]$ and $x_0 = 0$. In addition $x$ admits a L\'evy area above itself, that is, there exists a two index map $\mathbf{x}^2 : {\mathcal{S}_{2}\left([0,T]\right)} \to \R^{d,d}$ which belongs to $\cac_{2}^{2\ga}(\R^{d,d})$ and such that 
\begin{equation*}
\der {\mathbf{x}}_{sut}^{2;ij} = {\der x}_{su}^{i} \otimes {\der x}_{ut}^{j},
\quad\text{  and  }\quad
{\mathbf{x}_{st}^{2;ij} + \mathbf{x}_{st}^{2;ji}} =  \der x_{st}^{i} \otimes \der x_{st}^{j} .
\end{equation*}
The $\ga$-H\"older norm of $x$ is denoted by:
\begin{equation*}
\|\mathbf{x}\|_{\gamma} = \cn(x;\cac_1^{\ga}([0,T],\R^d))+\cn(\mathbf{x^2};\cac_2^{2\ga}([0,T],\R^{d,d})).
\end{equation*}
\end{hypothesis}

\noindent
Preparing the ground for the upcoming change of variable formula in Proposition~\ref{prop:ito_strat}, we now define the notion weakly controlled process as a slight variation of the usual one. 
\begin{definition}\label{def:weakly-ctrld}
Let $z$ be a process in $\cac_{1}^{\ga}(\R^n)$ with $1/3 < \ga \leq 1/2$ and consider $\eta > \ga$. We say that $z$ is weakly controlled by $x$ with a remainder of order $\eta$ if $\der z \in \cac_2^{\ga}(\R^n)$ can be decomposed into
\begin{equation*}
\der z^{i} = \zeta^{i i_1} \der x^{i_1} + r^{i}, ~~\text{   i.e.  }~~ {\der z}_{st}^{i} = \zeta_{s}^{i i_1}{\der x}_{st}^{i_1} + r_{st}^{i}
\end{equation*}
for all $(s,t) \in \mathcal{S}_2 \left([0,T]\right)$. In the previous formula we assume $\zeta \in \cac_{1}^{\eta - \ga}(\R^{n,d})$ and $r$ is a more regular remainder such that $r \in \cac_{2}^{\eta}(\R^{n})$. The space of weakly controlled paths will be denoted by $\mathcal{Q}_{\ga,\eta}(\R^n)$ and a process $z \in \mathcal{Q}_{\ga,\eta}(\R^n)$ can be considered as a couple $(z, \zeta)$. The natural semi-norm on $\mathcal{Q}_{\ga, \eta}(\R^n)$ is given by
\begin{equation*}
\cn[z;\mathcal{Q}_{\ga,\eta}(\R^n)] = \cn[z; \cac_1^{\ga}(\R^n)] + \cn[\zeta; \cac_1^{\infty}(\R^{n,d})] + \cn[\zeta; \cac_1^{\eta - \ga}(\R^{n,d})] + \cn[r; \cac_2^{\eta}(\R^n)].
\end{equation*}
\end{definition}

\noindent
Let ${\textnormal{Lip}}^{n+\lambda}$ denote the space of $n$-times differential functions with $\lambda-$H\"older $n$th derivative, endowed with the norm:
\begin{equation*}
{\|f\|}_{n, \lambda} = {\|f\|}_{\infty} + \sum_{k=1}^{n} {\|{\partial}^k f\|}_{\infty} + {\|{\partial}^n f\|}_{\lambda} .
\end{equation*}
The following gives a composition rule which asserts that our rough path $x$ composed with a $\textnormal{Lip}^{1+\lambda}$ function is weakly controlled.
\begin{proposition}\label{prop:smooth_of_rough}
Let $f:\R^d \to \R^n$ be a $\textnormal{Lip}^{1+\lambda}$ function and set $z = f(x)$. Then $z \in \mathcal{Q}_{\ga,\si}(\R^n)$ with $\si = \ga(\lambda + 1)$, where $\mathcal{Q}_{\ga,\si}(\R^n)$ is introduced in Definition \ref{def:weakly-ctrld}, and it can be decomposed into $\der z = \zeta \der x + r$, with 
\begin{equation*}
\zeta^{i i_1} = {\partial}_{i_1} f_i(x) ~~\text{ and }~~r^{i} = \der f_i(x) - {\partial}_{i_1} f_i(x){\der x}_{st}^{i_1}. 
\end{equation*} 
Furthermore, the norm of $z$ as a controlled process can be bounded as follows:
\begin{equation*}
\cn[z; \mathcal{Q}_{\ga, \si}] \leq K{\|f\|}_{1,\lambda} (1 + \cn^{1+\lambda}[x ; \cac_{1}^{\ga}(\R^d)]),
\end{equation*}
where K is a positive constant. 
\end{proposition}
\begin{proof}
The algebraic part of the assertion is straightforward. Just write
\begin{equation*}
{\der z}_{st} = f(x_t) - f(x_s) = {\partial}_{i_1} f(x_s) {\der x}_{st}^{i_1} + r_{st}
\end{equation*}
The estimate of $\cn[z; \mathcal{Q}_{\ga, \si}]$ is obtained from the estimates of $\cn[z; \cac_1^{\ga}(\R^n)]$, $\cn[\zeta; \cac_1^{\infty}(\R^{n,d})]$, $\cn[\zeta; \cac_1^{\sigma - \gamma}(\R^{n,d})]$ and $\cn[r; \cac_2^{\sigma}(\R^n)]$. The details are similar to \cite[Appendix]{Gubinelli} and left to the patient reader.
\end{proof}

\noindent
More generally, we also need to specify the composition of a controlled process with a $\textnormal{Lip}^{1+\lambda}$ function. The proof of this proposition is similar to Proposition~\ref{prop:smooth_of_rough} and omitted for sake of conciseness.
\begin{proposition}\label{prop:general_smooth_controlled}
Let $z \in \mathcal{Q}_{\ga,\si}(\R^n)$ with decomposition $\der z = \tilde{\zeta} \der x + \tilde{r}$ and $g:\R^n \to \R^m$ be a $\textnormal{Lip}^{1+\lambda}$ function. Set $w = g(x)$. Then $w \in \mathcal{Q}_{\ga,\si}(\R^m)$ with $\si = \ga(\lambda + 1)$ and it can be decomposed into $\der w = \zeta \der x + r$, with
\begin{equation*}
\zeta^{i i_1} = {\partial}_{i_2} f_i(x){\tilde{\zeta}}^{i_2, i_1} .
\end{equation*} 
\end{proposition}

The class of weakly controlled paths provides a natural and basic set of functions which can be integrated with respect to a rough path. The basic proposition in this direction, whose proof can be found in \cite{Gubinelli}, is summarized below.

\begin{theorem}\label{thm:int_controlled}
For $1/3 < \ga \leq 1/2$, let $x$ be a process satisfying Hypothesis~\ref{hyp:x}. Furthermore let $m\in \mathcal{Q}_{\ga, \eta}(\R^d)$ with $\eta + \gamma > 1$, whose decomposition is given by $m_0=b\in \R^d$ and
\begin{equation*}
\der m^{i} = \mu^{ii_1}\der x^{i_1} + r^{i}~~~\text{  where  }~~~\mu \in \cac_1^{\eta - \ga}(\R^{d,d}), r\in\cac_{2}^{\eta}(\R^n) .
\end{equation*}
Define $z$ by $z_0 = a \in \R^d$ and
\begin{equation*}
\der z = m^{i} \der x^i + \mu^{ii_1}{\mathbf{x}}^{2;i_1 i} - \Lambda(r^{i}\der x^{i} + \der \mu^{i i_1}{\mathbf{x}}^{2;i_1 i}) .
\end{equation*} 
Finally, set
\begin{equation*}
\mathcal{I}_{st} (m dx) = \int_{s}^{t} {\langle m_u, dx_u \rangle}_{\R^d} := {\der z}_{st}.
\end{equation*}
Then this integral extends Young integration and coincides with the Riemann-Stieltjes integral of $m$ with respect to $x$ whenever these two functions are smooth.
Furthermore, $\mathcal{I}_{st} (m dx)$ is the limit of modified Riemann sums:
\begin{equation*}\label{eq:integr_inft_sum}
\mathcal{I}_{st} (m dx) = \lim_{|\Pi_{st}|\to 0} \sum_{q = 0}^{n-1} [m_{t_q}^{i} {\der x}_{t_q t_{q+1}}^{i} + \mu_{t_q}^{i i_1}{\mathbf{x}}_{t_q t_{q+1}}^{2;i_1 i}] ,
\end{equation*}
for any $0 \leq s < t \leq T$, where the limit is taken over all partitions $\Pi_{st} = \lbrace s=t_0, \dots , t_n = t \rbrace$ of $[s,t]$, as the mesh of the partition goes to zero.
\end{theorem}

\subsection{It\^o-Stratonovich formula}
We now state a change of variable formula for a function $g(x)$ of a rough path, under minimal assumptions on the regularity of $g$. To the best of our knowledge, this proposition cannot be found in literature, and therefore a short and elementary  proof is included. The techniques of this proof will prove to be useful for the study of our system \eqref{eq:power-SDE} in the one-dimensional case.  
\begin{proposition}\label{prop:ito_strat}
Let $x$ satisfy Hypothesis~\ref{hyp:x}. Let $g$ be a $\textnormal{Lip}^{2+\lambda}$ function such that $(\lambda + 2)\ga > 1$. Then
\begin{equation}\label{eq:ito-strato-formula}
[\der (g(x))]_{st} = \mathcal{I}_{st} (\nabla g(x) dx) = \int_{s}^{t} {\langle \nabla g(x_u), dx_u \rangle}_{\R^d} ,
\end{equation}
where the integral above has to be understood in the sense of Theorem~\ref{thm:int_controlled}.
\end{proposition}

\begin{proof}
Consider a partition ${\Pi}_{st} = \lbrace s=t_0 < \cdots t_n =t \rbrace$ of $[s,t]$. The following identity holds trivially:
\begin{align} 
&g(x_t) - g(x_s) = \sum_{q=0}^{n-1} \left[ g(x_{t_{q+1}}) - g(x_{t_q}) \right] \nonumber\\ 
\label{eq:g_taylor}& = \sum_{q=0}^{n-1} \left[ \sum_{i}\partial_i g(x_{t_q}) {\der x}_{t_q t_{q+1}}^{i} + \dfrac{1}{2} \sum_{i_1, i_2} {\partial}_{i_1 i_2}^{2} g(x_{t_q}) {\der x}_{t_q t_{q+1}}^{i_1} {\der x}_{t_q t_{q+1}}^{i_2} + r_{t_q t_{q+1}} \right] 
\end{align}
where 
\begin{equation*}
r_{t_q t_{q+1}} = g(t_{q+1}) - g(t_q) - \sum_{i}\partial_i g(x_{t_q}) {\der x}_{t_q t_{q+1}}^{i} - \frac{1}{2} \sum_{i_1, i_2} {\partial}_{i_1 i_2}^{2} g(x_{t_q}) {\der x}_{t_q t_{q+1}}^{i_1} {\der x}_{t_q t_{q+1}}^{i_2}.
\end{equation*}
Furthermore, an elementary Taylor type argument shows that for all $i_1, i_2$ there exists an element ${\xi}_{i_1 i_2}^q$ of $[x_{t_q}, x_{t_{q+1}}]$ such that

\begin{align*}
r_{t_q t_{q+1}} &= \frac{1}{2} \sum_{i_1, i_2} {\partial}_{i_1 i_2}^{2} f({\xi}_{i_1 i_2}^q) {\der x}_{t_q t_{q+1}}^{i_1} {\der x}_{t_q t_{q+1}}^{i_2} - \frac{1}{2} \sum_{i_1,i_2} {\partial}_{i_1 i_2}^{2} f(x_{t_q}) {\der x}_{t_q t_{q+1}}^{i_1} {\der x}_{t_q t_{q+1}}^{i_2}\\
&= \frac{1}{2} \sum_{i_1, i_2} \left({\partial}_{i_1 i_2}^{2} f({\xi}_{i_1 i_2}^q) -  {\partial}_{i_1 i_2}^{2} f(x_{t_q})\right) {\der x}_{t_q t_{q+1}}^{i_1} {\der x}_{t_q t_{q+1}}^{i_2}.
\end{align*}

We now invoke the fact that $g \in \textnormal{Lip}^{2+\la}$ in order to get
\begin{equation*}
\left| r_{t_q t_{q+1}}\right| \leq C {|t_{q} - t_{q+1}|}^{(2+\lambda)\ga},
\end{equation*}
where $C$ is a constant depending on $g$ and $x$. Thus, since $(\lambda + 2)\ga > 1$, it is easily seen that 
\begin{equation}\label{eq:rem=0}
\lim_{|\Pi_{st}| \to 0} \sum_{q=0}^{n-1} r_{t_q t_{q+1}} = 0.
\end{equation}

\noindent 
In addition, using Hypothesis~\ref{hyp:x} and continuity of the partial derivatives, we can write
\begin{equation}\label{eq:simplified_2ndterm}
\frac{1}{2} \sum_{i_1, i_2} {\partial}_{i_1 i_2}^{2} f(x_{t_q}) {\der x_{t_q t_{q+1}}^{i_1}} {\der x_{t_q t_{q+1}}^{i_2}} = \sum_{i_1, i_2} {\partial}_{i_1 i_2}^{2} f(x_{t_q}) \mathbf{x}_{t_q t_{q+1}}^{\textbf{2};i_1 i_2} .
\end{equation}
Plugging \eqref{eq:rem=0} and \eqref{eq:simplified_2ndterm} into \eqref{eq:g_taylor} we get
\begin{equation}\label{eq:formula_part1}
g(x_t) - g(x_s) = \lim_{|\Pi_{st}| \to 0} \sum_{q=0}^{n-1} \partial_i g(x_{t_q}) {\der x}_{t_q t_{q+1}}^{i} + \sum_{q=0}^{n-1}{\partial}_{i_1 i_2}^{2} f(x_{t_q}) \mathbf{x}_{t_q t_{q+1}}^{\textbf{2};i_1 i_2}, 
\end{equation}
for all $(s,t) \in \mathcal{S}_2\left[0,T]\right)$.

\noindent
On the other hand looking at the decomposition of $\nabla g(x)$ as a weakly controlled process and using Proposition~\ref{prop:smooth_of_rough} we obtain:
\begin{equation*}
{\der \left[\nabla g(x)\right]}^{i}_{st} = {\der \partial_i g(x)}_{st} = {\partial}_{i_1 i}^2 g(x_s) {\der x}_{st}^{i_1} + R_{st}^{i},
\end{equation*}
where $R$ lies in $\cac_2^{(1+\lambda)\ga}$. Then using the Riemann sum representation \eqref{eq:integr_inft_sum} of rough integrals, we have
\begin{equation*}
\mathcal{I}_{st}(\nabla f(x) dx) = \lim_{|{\Pi}_{st}| \to 0} \left[ \sum_{q=0}^{n-1} \partial_i g(x_{t_q}) {\der x}_{t_q t_{q+1}}^{i} + \sum_{q=0}^{n-1}{\partial}_{i_1 i_2}^{2} f(x_{t_q}) \mathbf{x}_{t_q t_{q+1}}^{\textbf{2};i_1 i_2}  \right]. 
\end{equation*}
Comparing the above formula with \eqref{eq:formula_part1} proves the result. 
\end{proof}

\section{Differential equations: setting and one-dimensional case}

In this section we will give the general formulation and assumptions for equation~\eqref{eq:power-SDE}. Then we  state an existence result in dimension 1, which follows quickly from our preliminary considerations in Section \ref{sec:rough}.

\subsection{Setting}\label{subsec:setting}
Recall that we are considering the following rough differential equation:
\begin{equation}\label{eq:power_SDE}
y_t = a + \sum_{j=1}^{d} \int_{0}^{t} \sigma^j(y_s)dx_s^j,
\end{equation}
where $x$ satisfies Hypothesis~\ref{hyp:x} and $\si^1, \ldots, \si^d$ are vector fields on $\R^m$. In this section we will specify some general assumptions on the coefficient $\si$, which will prevail for the remainder of the article. 

\noindent
Let us start with a regularity assumption on $\sigma$:
\begin{hypothesis}\label{hyp:sigma+}
Let $F$ stand for either $\si$ or $D\si\cdot\si$. Let $\ka > 0$ be a constant such that $\ga + \ka > 1$, where $\ga$ is introduced in Hypothesis~\ref{hyp:x}. We assume that $F(0)=0$, and that for all $\xi_1, \xi_2 \in \R^m$ we have 
\begin{equation}\label{eq:sigma+_eqn}
\left|F({\xi}_1) - F({\xi}_2)\right| \lesssim \left| {|\xi_1|}^{\alpha} - {|\xi_2|}^{\alpha} \right|,
\end{equation}
where $\alpha = \ka$ if $F=\si$ and $\alpha=2\ka-1$ if $F=D\si\cdot\si$.

In addition to above, we assume that outside of a neighborhood of $0$, $\si$ behaves like a $\textnormal{Lip}_{loc}^{p}$ function with $p > \frac{1}{\ga}$, or in other words, $\si$ is bounded with bounded two derivatives and the second derivative is locally H\"older continuous with order larger than $(\frac{1}{\ga} - 2)$.
\end{hypothesis}

\noindent
We also need a more specific assumption in dimension 1:
\begin{hypothesis}\label{hyp:1-d_sigma+}
Whenever $d=1$, assume $\si$ is positive on $\R_{+}$ and that $\phi$ defined by $\phi(\xi) = \int_0^{\xi}\frac{ds}{\si(s)}$ exists. Also consider $\ka > 0$ as in Hypothesis~\ref{hyp:sigma+}. Then we assume for all $\xi_1, \xi_2 \in \R$ we have
\begin{equation*}
\left| F({\xi}_1) - F({\xi}_2) \right| \lesssim \left| {|\xi_1|}^{\frac{2\ka-1}{1-\ka} \wedge 1} - {|\xi_2|}^{\frac{2\ka-1}{1-\ka} \wedge 1} \right|,
\end{equation*}
where $F$ stands for the function $(D\si\cdot\si) \circ \phi^{-1}$.
\end{hypothesis}

\begin{remark} 
The hypotheses \ref{hyp:sigma+} and \ref{hyp:1-d_sigma+} above are true for a power coefficient of the form $\si(\xi)=c_1\left({|\xi|}^{\ka} \wedge c_2\right)$.
\end{remark}

\noindent
For the pairs of $F$ and $\alpha$ listed above, one can define a related seminorm as follows:
\begin{equation}\label{eq:sigma+_const}
{\mathcal{N}}_{\alpha, F} := \sup \left\lbrace \dfrac{|F(\xi_2) - F(\xi_1)|}{\left|{|\xi_2|}^{\alpha}-{|\xi_1|}^{\alpha}\right|} :|\xi_1| \neq |\xi_2| \right\rbrace
\end{equation}

\noindent
The following elementary lemma brings some useful estimates which will be used in Section~\ref{sec:multi-dim}. The reader is referred to \cite{YSDE} for its proof. 
\begin{lemma}\label{lem:cor_hyp_sigma+}
Assume $F$ satisfies \eqref{eq:sigma+_eqn}. Then 
\begin{equation*}
|F(\xi_2) - F(\xi_1)| \leq \dfrac{\alpha}{\alpha + \eta} \mathcal{N}_{\alpha, F} (|\xi_2|^{-\eta} + |\xi_1|^{-\eta}) |\xi_2 - \xi_1|^{\alpha + \eta} ,
\end{equation*}
for any $0 \leq \eta \leq 1-\alpha$ and $\xi_1, \xi_2 \in \R^m \setminus \lbrace 0 \rbrace$.
\end{lemma}

Finally we add some assumptions on the first and second order derivatives of $\si$, which will be mainly invoked in the proof of Theorem~\ref{thm:R_bnd}.
\begin{hypothesis}\label{hyp:si_der}
The derivatives of $\si$ satisfy the following:
\begin{equation}\label{eq:si_der}
|D \si(\xi)| \lesssim |\xi|^{\ka-1} ~\text{ and }~ |D^2 \si(\xi)| \lesssim |\xi|^{\ka-2} \text{ for }\xi \neq 0.
\end{equation}
\end{hypothesis}

\subsection{One-dimensional differential equations}\label{sec:1d}
In the one-dimensional case, similarly to what is done for more regular coefficients (See \cite{WY}), one can prove that a suitable function of $x$ solves equation \eqref{eq:power_SDE}. This stems from an application of our extension of It\^o's formula (see Proposition~\ref{prop:ito_strat}) and is obtained in the following theorem.
\begin{theorem}\label{thm:one_d}
Consider equation \eqref{eq:power_SDE} with $m=d=1$, let $\si:\R \to \R$ and assume Hypothesis~\ref{hyp:1-d_sigma+} to hold true. Assume $\ga \in \left(\frac{1}{3},\frac{1}{2}\right]$ and $\ka + \ga > 1$. Let $\phi$ be the function defined in Hypothesis~\ref{hyp:1-d_sigma+}. Then the function $y = \phi^{-1}(x + \phi(a))$ is a solution of the equation 
\begin{equation}\label{eq:power_SDE_1d}
y_t = a + \int_0^t \si(y_s)dx_s, ~~t \geq 0.
\end{equation}
\end{theorem}

\begin{proof}
Let $\psi(\xi) = {\phi}^{-1}(\xi + \phi(a))$. Due to the definition of $\phi$, some elementary computations show that ${\psi}'(\xi) = \frac{1}{{\phi}'({\phi}^{-1}(\xi + \phi(a)))} = \si(\psi(\xi))$ and thus we are reduced to show
\begin{equation}\label{eq:reduced_to_solve}
\delta \psi(x)_{st} = \int_{s}^{t} \psi'(x_u) dx_u.
\end{equation} 
\noindent
To this aim, observe that the second derivative of $\psi$ satisfies
\begin{equation*}
\psi''(\xi) = D\si(\psi(\xi))\psi'(\xi) = (D\si\cdot\si)(\psi(\xi)).
\end{equation*}
 Using Hypothesis~\ref{hyp:1-d_sigma+}, $\psi''$ is thus $\lambda-$H\"older continuous where $\lambda = \frac{2\ka-1}{1-\ka} \wedge 1$, that is, $\psi$ is a $\textnormal{Lip}^{2+\lambda}$ function. Moreover, since $\ka + \ga > 1$ and $\ga \in\left(\frac{1}{3},\frac{1}{2}\right]$ we find $(\lambda + 2)\ga > 1$. Consequently we can invoke Proposition~\ref{prop:ito_strat} and hence we obtain directly \eqref{eq:reduced_to_solve}. The result is now proved.
\end{proof}

\begin{remark}
If $a=0$, we do not have uniqueness of solution since in addition to the solution defined above, $y \equiv 0$ solves equation \eqref{eq:power_SDE_1d}. This is not in contradiction to the results stated in \cite{Cherny-Engelbert} where the authors deal with equations with non-vanishing coefficients. In our case, $\si(0)=0$.
\end{remark}

\begin{remark}
As the reader might see, Theorem \ref{thm:one_d} is an easy consequence of the change of variable formula \eqref{eq:ito-strato-formula}. This is in contrast with the corresponding proof in \cite{YSDE}, which relied on a negative moment estimate and non trivial extensions of Young's integral in the fractional calculus framework.
\end{remark}

\section{Multidimensional Differential Equations}\label{sec:multi-dim}
In the multidimensional case, our strategy in order to construct a solution is based (as in \cite{YSDE}) on quantifying an additional smoothness of the solution $y$ as it approaches the origin. However, our computations here are more involved than in \cite{YSDE}, due to the fact that we are handling a rough process $x$.

\subsection{Prelude}
In this section, we will introduce a sequence of stopping times, similarly to~\cite{YSDE}. We assume that each component $\si^j:\R^m \to \R^m$ satisfies Hypothesis~\ref{hyp:sigma+} and we consider the following equation for a fixed $a \in \R^m\setminus\lbrace 0 \rbrace$:
\begin{equation}\label{eq:sde-power}
y_t = a + \sum_{j=1}^{d}\int_0^t\si^j(y_u) \, dx_u^j, \quad t \in [0,T] ,
\end{equation}
where $T>0$ is a fixed arbitrary horizon and $\mathbf{x}=(x,\mathbf{x^2})$ is a $\gamma$-rough path above $x$, as given in Hypothesis \ref{hyp:x}.

Our considerations start from the fact that, as long as we are away from 0, we can solve equation \eqref{eq:sde-power} as a rough path equation with regular coefficients. In particular the following can be shown under the above set-up. See \cite{FV-bk}.

\begin{theorem}  \label{prop:th1} 
Assume Hypothesis~\ref{hyp:sigma+} is fulfilled. Then there exists a continuous function $y$ defined on $[0,T]$ and an instant $\tau \le T$, such that one of the following two possibilities holds:  

\begin{itemize}
\item[(A)] $\tau =T$, $y$ is non-zero on $[0,T]$,  $y\in\cac^{\ga}([0,T];\R^m)$   and $y$ solves equation  \eqref{eq:sde-power} on $[0,T]$, where the integrals $\int \si^{j}(y_{u}) \, dx_{u}^{j}$ are understood in the rough path sense.

\item[(B)]   We have $\tau<T$. Then for any $t<\tau$, the path $y$ sits in $\cac^{\ga}([0,t];\R^m)$   and $y$ solves equation  \eqref{eq:sde-power} on $[0,t]$.
Furthermore,  $y_s \not=0$ on   $[0,\tau)$, $\lim_{t\rightarrow \tau} y_t=0$ and $y_t=0$ on the interval $[\tau,T]$.
\end{itemize}
\end{theorem}

Option (A) above leads to classical solutions of equation \eqref{eq:sde-power}. 
In the rest of this section, we will assume (B), that is the function $y$ given by Proposition~\ref{prop:th1}  vanishes  in the interval $[\tau, T]$. The aim of this section is to prove the following:
\begin{itemize}
\item 
The path $y$ is globally $\ga$-H\"older continuous on $[0,T]$.

\end{itemize}
To achieve this we will require some additional hypotheses on $x$ (See Hypothesis~\ref{hyp:reg-x-gamma-gamma1} below). 

Quantification of the increased smoothness of the solution as it approaches the origin would require a partition of the interval $(0,\tau]$ as follows. Let $a_j=2^{-j}$ and consider the following decomposition of $\R_{+}$:
\begin{equation*}
\R_{+} = \bigcup\limits_{j=-1}^{\infty} I_j,
\end{equation*}
where
\begin{equation*}
I_{-1}=\left[1,\infty\right), \quad {\rm and} \quad  I_{q}=[a_{q+1}, a_{q}),  \quad q\geq 0.
\end{equation*}
Also consider:
\begin{equation*}
J_{-1}= \left[3/4, \infty\right), \quad  {\rm and} \quad
J_{q} = \lc \frac{a_{q+2}+a_{q+1}}{2}, \frac{a_{q+1}+a_{q}}{2}  \rp
=: \lc \hat{a}_{q+1}, \hat{a}_{q} \rp, \quad q\ge 0.
\end{equation*}
Observe that owing to the definition of $a_q$, we have $\hat{a}_{q} = \frac{3}{2^{q+2}}$. Let $q_0$ be such that $a \in I_{q_0}$. Define $\la_0 = 0$ and 
\begin{equation*}
\tau_0=\inf\lbrace t\geq0 : |y_t| \not\in I_{q_0} \rbrace
\end{equation*}
By definition, $y_{\tau_0} \in J_{\hat{q}_0}$ with ${\hat{q}}_0 \in \lbrace q_0, q_0 - 1 \rbrace$. Now define 
\begin{equation*}
\la_1 = \inf \lbrace t \geq \tau_0:|y_t| \not\in J_{{\hat q}_{0}}\rbrace
\end{equation*}
Thus we get a sequence of stopping times $\la_{0}<\tau_{0}<\cdots<\la_{k}<\tau_{k}$, such that
\begin{equation}\label{eq:def-sigma-tau-k}
y_{t} \in \lc \frac{b_1}{2^{q_{k}}}, \, \frac{b_2}{2^{q_{k}}} \rc,
\quad\text{for}\quad
t\in [\la_{k},\tau_{k}] \cup [\tau_{k},\la_{k+1}],
\end{equation}
where $b_1=\frac 38$, $b_2=\frac 34$
and $q_{k+1}=q_{k}+\ell$, with $\ell\in\{-1,0,1\}$, for $q_k\ge 1$. If $q_k =0$ or $q_k=1$, then we can choose the upper bound $b_2$ as $b_2=\infty$.

\begin{remark}
Since this problem relies heavily on radial variables in $\R^{m}$, we alleviate vectorial notations and carry out the computations below for $m=d=1$. Generalizations to higher dimensions are straight forward.
\end{remark}

\subsection{Regularity estimates}\label{sec:reg-estimates}
Let $\pi = \lbrace 0=t_0 < t_1 < \cdots < t_{n-1} < t_n=T \rbrace$ be a partition of the interval $[0,T]$ for $n \in \N$. Denote by $\cac_2(\pi)$ the collection of functions $R$ on $\pi$ such that $R_{t_k t_{k+1}} = 0$ for $k=0,1,\dots n-1$. We now introduce some operators on discrete time increments, which are similar to those in Section \ref{sec:rough}. First, we define the operator $\der:\cac_2(\pi) \to \cac_3(\pi)$ by
\begin{equation}\label{eq:der_def}
\der R_{sut} = R_{st} - R_{su} -R_{ut}~\text{ for }s,u,t \in \pi
\end{equation}  
The H\"older seminorms we will consider are similar to those introduced in \eqref{eq:norm_1} and \eqref{eq:norm_2}. Namely, for $R \in \cac_2(\pi)$ we set
\begin{equation*}
{\|R\|}_{\mu} = \sup_{u,v \in \pi} \dfrac{R_{uv}}{{|u-v|}^{\mu}}~ \text{ and } ~{\|\der R\|}_{\mu} = \sup_{s,u,t \in \pi} \dfrac{|\der R_{sut}|}{{|t-s|}^{\mu}}
\end{equation*}

\noindent
We now state a sewing lemma for discrete increments which is similar to \cite[Lemma 2.5]{Liu-T}. Its proof is included here for completeness.
\begin{lemma}\label{lem:discr_sewing}
For $\mu > 1$ and $R \in \cac_2(\pi)$, we have 
\begin{equation*}
{\|R\|}_{\mu} \leq K_{\mu} {\|\der R\|}_{\mu},
\end{equation*}
where $K_{\mu} = 2^{\mu} \sum_{l=1}^{\infty} \frac{1}{l^{\mu}}$
\end{lemma}

\begin{proof}
Consider some fixed $t_i, t_j \in \pi$. Since $R \in \cac_2(\pi)$ we have $\sum_{k=i}^{j-1} R_{t_k t_{k+1}} = 0$.
Hence, for an arbitrary sequence of partitions $\lbrace \pi_l; 1 \leq l \leq j-i-1 \rbrace$, where each $\pi_l$ is a subset of $\pi \cap \left[t_i, t_j\right]$ with $l+1$ elements, we can write (thanks to a trivial telescoping sum argument):
\begin{equation}\label{eq:R_sum}
R_{t_i t_j} = R_{t_i t_j} - \sum_{k=i}^{j-1} R_{t_k t_{k+1}} = \sum_{l=1}^{j-i-1} (R^{\pi_l} - R^{\pi_{l+1}}),
\end{equation}
where we have set $R^{\pi_l} = \sum_{k=0}^{l-1} R_{t_k^l t_{k+1}^l}$.
We now specify the choice of partitions $\pi_l$ recursively: 

\noindent
Define $\pi_{j-i} = \pi \cap [t_i, t_j]$. Given a partition $\pi_l$ with $l+1$ elements, $l=2,\dots,j-i$, we can find $t_{k_l}^l \in \pi_l \setminus \lbrace t_i, t_j \rbrace$ such that 
\begin{equation}\label{eq:t_diff}
t_{k_l +1}^l - t_{k_l - 1}^l \leq \dfrac{2(t_j - t_i)}{l}.
\end{equation}
Denote by $\pi_{l-1}$ the partition $\pi_l \setminus \lbrace t_{k_l}^l \rbrace$. Owing to \eqref{eq:der_def}, we obtain:
\begin{equation*}
\left| R^{\pi_{l-1}} - R^{\pi_l} \right| = \left| \der R_{t_{k_l-1}^l t_{k_l}^l t_{k_l+1}^l} \right| 
 \leq {\|\der R\|}_{\mu} (t_{k_l + 1}^l - t_{k_l - 1}^l)^{\mu} 
 \leq {\|\der R\|}_{\mu} \dfrac{2^{\mu}(t_j - t_i)^{\mu}}{l^{\mu}},
\end{equation*}
where the second inequality follows from \eqref{eq:t_diff}. Now plugging the above estimate in \eqref{eq:R_sum} we get
\begin{equation*}
\left| R_{t_i t_j} \right| \leq 2^{\mu} (t_j - t_i)^{\mu} {\|\der R\|}_{\mu} \sum_{l=1}^{j-i-1} \dfrac{1}{(l+1)^{\mu}} 
\leq K_{\mu} (t_j - t_i)^{\mu} {\|\der R\|}_{\mu} .
\end{equation*}
By dividing both sides by $(t_i - t_j)^{\mu}$ and taking supremum over $t_i, t_j \in \pi$, we obtain the desired estimate.
\end{proof}

\noindent
Next we define an increment $R$ which is obtained as a remainder in rough path type expansions.
\begin{definition}
Let $y$ and $\tau$ be defined as in Proposition~\ref{prop:th1}. For $(s,t) \in \mathcal{S}_2\left(\left[0,\tau\right]\right)$, let $R_{st}$ be defined by the following decomposition:
\begin{equation}\label{eq:decomp}
{\delta y}_{st} = \si(y_s){\delta x}_{st} + (D\si \cdot \si)(y_s)\mathbf{x}_{st}^{\mathbf{2}} + R_{st}.
\end{equation} 
\end{definition}

The theorem below quantifies the regularity improvement for the solution $y$ of equation~\eqref{eq:sde-power} as it gets closer to $0$.
\begin{proposition}\label{thm:R_bnd}
Consider a rough path $x$ satisfying Hypothesis \ref{hyp:x}.
Assume $\sigma$ and $(D\si \cdot \si)$ follows Hypothesis~\ref{hyp:sigma+} (and thus the subsequent Lemma \ref{lem:cor_hyp_sigma+}). Also assume Hypothesis~\ref{hyp:si_der} holds. Then there exist constants $c_{0,x}$, $c_{1,x}$ and $c_{2,x}$ such that for $s,t \in [\la_k, \la_{k+1})$ satisfying $|t-s| \leq c_{0,x} 2^{-\alpha q_k}$, with $\alpha := \frac{1 -\ka}{\ga}$, we have the following bounds:
\begin{equation}\label{eq:y_ga_norm}
\cn\left[y;\cac_1^{\ga}\left([s,t]\right)\right] \leq c_{1,x} 2^{-\ka q_k}
\end{equation}
and
\begin{equation}\label{eq:R_3ga_norm}
\cn\left[R;\cac_2^{3\ga}\left([s,t]\right)\right] \leq c_{2,x} 2^{(2-3\ka)q_k}.
\end{equation}
\end{proposition}

\begin{proof}
We divide this proof in several steps.

\noindent
\emph{Step 1: Setting.}
Consider the dyadic partition on $[s,t]$. Specifically, we set
\begin{equation*}
[\![s,t]\!]=\lcl t_i: t_i = s + \frac{i(t-s)}{2^n};i=0,\cdots,2^n\rcl
\end{equation*}
for all $n \in \mathbb{N}$. Define $y^n$ on $[\![s,t]\!]$ by setting $y_s^n = y_s$, and 
\begin{equation*}
\delta y^n_{t_i t_{i+1}} = \si(y_{t_i}^n)\delta x_{t_i t_{i+1}} + (D\si \cdot \si)(y_{t_i}^n)\mathbf{x}_{t_i t_{i+1}}^{\textbf{2}}
\end{equation*}
We also introduce a discrete type remainder $R^{n}$, defined for all $(u,v) \in \mathcal{S}_2\left( [\![s,t]\!] \right)$, as follows:
\begin{equation*}
R_{uv}^{n} = \delta y^n_{uv} - \si(y_{su}^n){\delta x}_{uv} - (D\si\cdot \si)(y_u^n)\mathbf{x}_{uv}^{\textbf{2}}
\end{equation*}

\noindent
Since $\ga > 1/3$ and $\sigma$ is sufficiently smooth away from zero, a second order expansion argument (see \cite[Section 10.3]{FV-bk}) shows that $\der y_{st}^n$ converges to $\der y_{st}$.

\noindent
\emph{Step 2: Induction hypothesis.}
Recall that we are working in $[{\lambda}_k, {\lambda}_{k+1})$. Hence, using \eqref{eq:def-sigma-tau-k} we can choose $n$ large enough so that 
\begin{equation}\label{eq:y^n_range}
y_u^n \in \left[\frac{a_1}{2^{q_k}}, \frac{a_2}{2^{q_k}}\right] ~\text{ for }~ u \in [\![s,t]\!], 
\end{equation}
where $a_1=\frac{2}{8}$ and $a_2=\frac{7}{8}$. In addition, using Hypothesis~\ref{hyp:sigma+}, \eqref{eq:sigma+_const} and \eqref{eq:y^n_range} above, we also have
\begin{equation}\label{eq:bnd1}
|\si(y_u^n)| \leq {\cn}_{\ka, \si} {|y_u^n|}^{\ka} \leq {\cn}_{\ka, \si} \left(\dfrac{a_2}{2^{q_k}}\right)^{\ka}
\end{equation}
as well as:
\begin{equation}\label{eq:bnd2}
|(D \si \cdot \si)(y_u^n)| \leq {\cn}_{2\ka -1, {D\si \cdot \si}} {|y_u^n|}^{2\ka - 1} \leq {\cn}_{2\ka - 1, D\si \cdot \si} \left(\dfrac{a_2}{2^{q_k}}\right)^{2 \ka - 1}.
\end{equation}
We now assume that $s$ and $t$ are close enough, namely for a given constant $c_0>0$, we have
\begin{equation}\label{eq:t,s_assumption}
|t-s| \leq c_0 2^{-\alpha q_k} = T_0.
\end{equation}
We will proceed by induction on the points of the partition $t_i$. That is, for $q \leq 2^{n}-1$ we assume that $R^n$ satisfies the following relation:
\begin{equation}\label{eq:R^n_assumption}
\cn[R^n ; \cac_2^{3\ga}[\![s,t_q]\!]] \leq c_2 2^{(2 - 3 \ka)q_k} 
\end{equation}
where $c_2$ is a constant to be fixed later. We will try to propagate this induction assumption to $[\![s,t_{q+1}]\!]$.

\noindent
\emph{Step 3: \textit{A priori} bounds on $y^n$.}
For $(u,v) \in \mathcal{S}_{2}\left( [\![s,t_q]\!] \right)$ we have: 
\begin{equation}\label{eq:y^n_decomp}
{\delta y}_{uv}^n = \si(y^n_u) {\delta x}_{uv} + (D \si \cdot \si)(y_u^n)\mathbf{x}_{uv}^{\textbf{2}} + R_{uv}^n.
\end{equation}
Hence, using \eqref{eq:bnd1}, \eqref{eq:bnd2} and our induction assumption \eqref{eq:R^n_assumption} we get:
\begin{align*}
\cn[y^n; \cac_{1}^{\ga} [\![s,t_q]\!]] \leq {{\cn}_{\ka, \si}} {\lp \dfrac{a_2}{2^{q_k}}\rp}^{\ka} \|\mathbf{x}\|_{\ga} + {{\cn}_{2 \ka -1, D \si \cdot \si}} \left(\frac{a_2}{2^{q_k}}\right)^{2 \ka -1} \|\mathbf{x}\|_{\ga} {|t_q - s|}^{\ga}\\
 + \cn[R^n; \cac_2^{3 \ga}[\![s,t]\!]]{|t_q - s|}^{2\ga}
\end{align*}
Since $|t_q -s| \leq T_0 = c_0 2^{-\alpha q_k}$, we thus have
\begin{align*}
\cn[y^n; \cac_{1}^{\ga} [\![s,t_q]\!]] \leq {{\cn}_{\ka, \si}} {\lp \dfrac{a_2}{2^{q_k}}\rp}^{\ka} \|\mathbf{x}\|_{\ga} + {{\cn}_{2 \ka -1, D \si \cdot \si}} \left(\frac{a_2}{2^{q_k}}\right)^{2 \ka -1} \|\mathbf{x}\|_{\ga} {\lp c_0 2^{-\alpha q_k} \rp}^{\ga} \\
                   + \cn[R^n; \cac_2^{3 \ga}[\![s,t]\!]]{ \lp c_0 2^{-\alpha q_k} \rp}^{2\ga}.
\end{align*}
Therefore taking into account the fact that $\alpha = \dfrac{1-\ka}{\gamma}$ and our assumption \eqref{eq:R^n_assumption}, we obtain:
\begin{equation}\label{eq:y^n_bnd}
\cn[y^n; \cac_{1}^{\ga} [\![s,t_q]\!]] \leq \tilde{c}~ 2^{-\ka q_k}
\end{equation}
where the constant $\tilde{c}$ is given by:
\begin{equation}\label{eq:tilde_c}
\tilde{c} = {{\cn}_{\ka, \si}} a_2^{\ka} \|\mathbf{x}\|_{\ga} + {{\cn}_{2 \ka -1, D \si \cdot \si}} a_2^{2 \ka -1}c_0^{\ga}\|\mathbf{x}\|_{\ga} + c_2 c_0^{2\ga}.
\end{equation} 

\noindent
\emph{Step 4: Induction propagation.}
Recall that $R_{uv}^n = \der y^n_{uv} - \si(y_{su}^n){\delta x}_{uv} - (D\si\cdot \si)(y_u^n)\mathbf{x}_{uv}^{\textbf{2}}$. Hence invoking Proposition~\ref{prop:der_rules} we have:
\begin{equation}\label{eq:der_R^n}
\der R_{uvw}^n = \mathcal{A}_{uvw}^{n,1}+\mathcal{A}_{uvw}^{n,2}+\mathcal{A}_{uvw}^{n,3},
\end{equation}
with 
\begin{equation*}
\mathcal{A}_{uvw}^{n,1} = -{\der \si(y^n)}_{uv}{\delta x}_{vw},\qquad \mathcal{A}_{uvw}^{n,2} = - \delta((D \si \cdot \si)(y^n))_{uv}\mathbf{x}_{vw}^{\textbf{2}} 
\end{equation*}
and
\begin{equation*}
\mathcal{A}_{uvw}^{n,3} = (D\si \cdot \si)(y_u^n){\delta \mathbf{x}}_{uvw}^{\mathbf{2}}.
\end{equation*}
We now treat those terms separately. The term $\mathcal{A}_{uvw}^{n,1}$ in \eqref{eq:der_R^n} can be expressed using Taylor expansion, which yields
\begin{equation*}
\mathcal{A}_{uvw}^{n,1} = -\left( D\si(y_u^n)\delta y_{uv}^n + \frac{1}{2} D^2 \si(\xi^n) \left({\delta y}_{uv}^n\right)^2 \right) {\delta x}_{vw},
\end{equation*}
for some $\xi^n \in [y_u^n, y_v^n]$. Now, using \eqref{eq:y^n_decomp} the above becomes
\begin{align}
\mathcal{A}_{uvw}^{n,1} =& - D\si(y_u^n) \left(\si(y^n_u) {\delta x}_{uv} + (D \si \cdot \si)(y_u^n)\mathbf{x}_{uv}^{\textbf{2}} + R_{uv}^n\right){\delta x}_{vw}
 - \frac{1}{2} D^2 \si(\xi^n) \left({\delta y}_{uv}^n\right)^2 {\delta x}_{vw}\nonumber \\
 =& -(D\si \cdot \si)(y_u^n) \der x_{uv} \der x_{vw} - D\si(y_u^n) (D\si \cdot \si)(y_u^n) \mathbf{x}_{uv}^{\mathbf{2}} \der x_{vw} \nonumber \\
 \label{eq:first_term}&\hspace{2in}- D\si(y_u^n) R_{uv}^n \der x_{vw} - \dfrac{1}{2} D^2 \si(\xi^n) (\der y_{uv}^n)^2 \der x_{vw}   .
\end{align}
Due to Hypothesis~\ref{hyp:x}, the first term of \eqref{eq:first_term} cancels $\mathcal{A}_{uvw}^{n,3}$ in \eqref{eq:der_R^n}. Therefore we end up with:
\begin{equation*}
\mathcal{A}_{uvw}^{n,1}+\mathcal{A}_{uvw}^{n,3} = - D\si(y_u^n) (D\si \cdot \si)(y_u^n) \mathbf{x}_{uv}^{\mathbf{2}} \der x_{vw} - D\si(y_u^n) R_{uv}^n \der x_{vw} - \dfrac{1}{2} D^2 \si(\xi_w^n) (\der y_{uv}^n)^2 \der x_{vw} .
\end{equation*}

Taking into account \eqref{eq:sigma+_eqn} and \eqref{eq:si_der} (similarly to what we did for \eqref{eq:bnd1}--\eqref{eq:bnd2}), as well as Hypothesis~\ref{hyp:x} and relation \eqref{eq:t,s_assumption} for $|t-s|$, plus the induction \eqref{eq:R^n_assumption} on $R^n$, we easily get:
\begin{multline}\label{eq:A_n1+A_n2}
\mathcal{A}_{uvw}^{n,1}+\mathcal{A}_{uvw}^{n,3} \leq \bigg\lbrace\left(\dfrac{a_1}{2^{q_k}}\right)^{\ka-1}\left(\dfrac{a_2}{2^{q_k}}\right)^{2\ka - 1}{\|\mathbf{x}\|}_{\ga}^2 + \left(\frac{a_1}{2^{q_k}}\right)^{\ka - 1} \|\mathbf{x}\|_{\ga} T_0^{\ga} \cn[R^n;\cac_2^{3 \ga}[\![s,t_q]\!]]\\ 
+ \frac{1}{2}\left(\frac{a_1}{2^{q_k}}\right)^{\ka - 2} \|\mathbf{x}\|_{\ga} {\cn[y^n;\cac_1^{\ga}[\![s,t_q]\!]]}^2 \bigg\rbrace {|w-u|}^{3\ga}.
\end{multline}

We are now left with the estimation of $\mathcal{A}^{n,2}$. To bound this last term we first use Lemma~\ref{lem:cor_hyp_sigma+} to get, for any $\eta \leq 2(1-\ka)$ 
\begin{equation*}
\left|\mathcal{A}_{uvw}^{n,2}\right| \leq \frac{2 \ka -1}{2 \ka -1 + \eta} \cn_{2\ka-1,D\si\cdot\si}(|y_u^n|^{-\eta}+|y_v^n|^{-\eta})|y_v^n - y_u^n|^{2 \ka -1 + \eta} \|\mathbf{x}\|_{\ga} |w-v|^{2 \ga},
\end{equation*}
which invoking \eqref{eq:y^n_range} and the definition of ${\cn[y^n;\cac_1^{\ga} [\![s,t_q]\!]]}$, yields:
\begin{multline*}
\left|\mathcal{A}_{uvw}^{n,2}\right| \leq \frac{2(2 \ka -1)}{2 \ka -1 + \eta} \cn_{2\ka-1,D\si\cdot\si}{\left(\frac{2^{q_k}}{a_1}\right)}^{\eta}\\ 
\times {\cn[y^n;\cac_1^{\ga} [\![s,t_q]\!]]}^{2\ka-1+\eta} {|v-u|}^{\ga(2\ka-1+\eta)}\|\mathbf{x}\|_{\ga}|w-v|^{2\ga}.
\end{multline*}
Finally using \eqref{eq:t,s_assumption} and the a priori bound on $y^n$ stated in \eqref{eq:y^n_bnd} we get:
\begin{multline}\label{eq:bnd_A^n2}
\left|\mathcal{A}_{uvw}^{n,2}\right| \leq \frac{2(2 \ka -1)}{2 \ka -1 + \eta} \cn_{2\ka-1,D\si\cdot\si}\left(\frac{a_1}{2^{q_k}}\right)^{-\eta}\\ 
\times\tilde{c}^{2 \ka -1 + \eta} 2^{-\ka(2\ka -1 + \eta)q_k}\|\mathbf{x}\|_{\ga} (c_0 2^{-\alpha q_k})^{\ga(2\ka - 1+\eta) - \ga}{|w-u|}^{3\ga}.
\end{multline}
Let us choose $\eta = 2(1-\ka)$. In this case we obviously have $2\ka -1 + \eta = 1$, and inequality~\eqref{eq:bnd_A^n2} can be recast as:
\begin{equation}\label{eq:bnd_A^n2-eta}
\left|\mathcal{A}_{uvw}^{n,2}\right| \leq \frac{2(2 \ka -1)}{a_1^{\eta}}  \cn_{2\ka-1,D\si\cdot\si}~  \tilde{c}~\|\mathbf{x}\|_{\ga} 2^{(2-3\ka)q_k} {|w-u|}^{3\ga}.
\end{equation}
We can now plug \eqref{eq:A_n1+A_n2} and \eqref{eq:bnd_A^n2-eta} back into \eqref{eq:der_R^n} in order to get:
\begin{multline*}
\cn[\delta R^n; \cac_3^{3 \ga}[\![s,t_{q+1}]\!]] \leq \left(\dfrac{a_1}{2^{q_k}}\right)^{\ka-1}\left(\dfrac{a_2}{2^{q_k}}\right)^{2\ka - 1}{\|\mathbf{x}\|}_{\ga}^2 + \left(\frac{a_1}{2^{q_k}}\right)^{\ka - 1} \|\mathbf{x}\|_{\ga} T_0^{\ga} \cn[R^n;\cac_2^{3 \ga}[\![s,t_q]\!]]\\ + \frac{1}{2}\left(\frac{a_1}{2^{q_k}}\right)^{\ka - 2} \|\mathbf{x}\|_{\ga} {\cn[y^n;\cac_1^{\ga}[\![s,t_q]\!]]}^2 \nonumber + \left(\frac{2(2 \ka -1)}{a_1^{\eta}}  \cn_{2\ka-1,D\si\cdot\si} ~ \tilde{c} \|\mathbf{x}\|_{\ga}\right)~2^{(2-3\ka)q_k}.
\end{multline*}
Therefore, thanks to our induction assumption \eqref{eq:R^n_assumption} and the a priori bound \eqref{eq:y^n_bnd}, the above becomes

\begin{equation*}
\cn[\delta R^n; \cac_3^{3 \ga}[\![s,t_{q+1}]\!]] \leq d 2^{(2-3\ka)q_k}
\end{equation*}
with 
\begin{equation}\label{eq:d}
d = \left(a_1^{\ka -1}a_2^{2\ka-1}\|\mathbf{x}\|_{\ga}^2 +  a_1^{\ka -1}\|\mathbf{x}\|_{\ga}c_0^{\ga}c_2  + \frac{1}{2}a_1^{\ka -2}{\tilde{c}}^2 {\|\mathbf{x}\|}_{\gamma}+ \frac{2(2 \ka -1)}{a_1^{\eta}} \cn_{2\ka-1,D\si\cdot\si} ~\tilde{c}\|\mathbf{x}\|_{\ga}\right)
\end{equation}
\noindent
Then using the discrete sewing Lemma~\ref{lem:discr_sewing}, we obtain
\begin{equation}\label{eq:R^n_ineq}
\cn[R^n ; \cac_2^{3 \gamma}[\![s,t_{q+1}]\!]] \leq K_{3\ga}\cn[\delta R^n; \cac_3^{3 \ga}[\![s,t_{q+1}]\!]] \leq \hat{c} 2^{(2-3\ka)q_k},
\end{equation}
where $K_{3\ga} = \sum_{l=1}^{\infty} \frac{1}{l^{3 \ga}}$ and $\hat{c}=d K_{3\ga}$.

Plugging in the value of $\tilde{c}$ from \eqref{eq:tilde_c} in the expression for $d$ in \eqref{eq:d} we find that $\hat{c}$ can be decomposed as
\begin{equation*}
\hat{c} = d K_{3\ga} = (d_{1,x} + d_{2,x})K_{3\ga},
\end{equation*}
where
\begin{equation*}
d_{1,x} = \left(a_1^{\ka -1}a_2^{2\ka-1}\|\mathbf{x}\|_{\ga}^2 + \frac{1}{2}a_1^{\ka -2}{{\cn}^2_{\ka, \si}} a_2^{2\ka} \|\mathbf{x}\|_{\ga}^3+ \frac{2(2 \ka -1)}{a_1^{\eta}} \cn_{2\ka-1,D\si\cdot\si} {\cn}_{\ka, \si} a_2^{\ka} \|\mathbf{x}\|^2_{\ga}\right).
\end{equation*}
and $d_{2,x}$ consist of terms containing positive powers of $c_0$, where we recall that $c_0$ is defined by \eqref{eq:t,s_assumption}. 

Looking at inequality \eqref{eq:R^n_ineq}, we need $\hat{c}$ to be less than $c_2$ in order to complete the induction propagation. Let us now fix $c_2=\frac{3}{2}d_{1,x} K_{3\ga}=c_{2,x}$ and choose $c_0 = c_{0,x}$ small enough so that $d_{2,x} < \frac{d_{1,x}}{2}$. This implies $\hat{c} = d K_{3\ga} = (d_{1,x} + d_{2,x})K_{3\ga} < \frac{3}{2}d_{1,x} K_{3\ga}=c_{2,x}$, which is what we required. Our propagation is hence established.

\noindent
\emph{Step 5: Conclusion.}
Completing the iterations over $t_q$ in $[\![s,t]\!]$ we get that relation \eqref{eq:R^n_assumption} is valid for $\cn[R^n;\cac_3^{3 \ga}[\![s,t]\!]]$. Next, put the values of $c_{0,x}$ and $c_{2,x}$ in $\tilde{c}$ as defined in \eqref{eq:y^n_bnd} and call this new value $c_{1,x}$. We thus get the following uniform bound over $n$:
\begin{equation*}
\cn[y^n; \cac_{1}^{\ga} [\![s,t]\!]] \leq c_{1,x} 2^{-\ka q_k}.
\end{equation*}
Our claims \eqref{eq:R_3ga_norm} and \eqref{eq:y_ga_norm} are now achieved by taking limits over $n$.
\end{proof}

In order to further analyze the increments of $y^n$, we need to increase slightly the regularity assumptions on $x$. This is summarized in the following hypothesis: 
\begin{hypothesis}\label{hyp:reg-x-gamma-gamma1}
There exists $\ep_{1}>0$ such that for $\ga_{1}=\ga+\ep_{1}$, we have $\|\mathbf{x}\|_{\ga_{1}}<\infty$.
\end{hypothesis}

The extra regularity imposed on $\mathbf{x}$ allows us to improve our estimates on remainders (in rough path expansions) in the following way.

\begin{proposition}\label{prop:regularity-gain}
Let us assume that Hypothesis~\ref{hyp:reg-x-gamma-gamma1} holds, as well as Hypothesis \ref{hyp:sigma+} and Hypothesis \ref{hyp:si_der}. For $k \geq 0$, consider $(s,t) \in \mathcal{S}_2\left([\la_k, \la_{k+1})\right)$ such that $|t-s| \leq c_{0,x} 2^{-\alpha q_k}$, where $c_{0,x}$ is defined in Theorem~\ref{thm:R_bnd}. Then the following second order decomposition for $\der y$ is satisfied:
\begin{equation}\label{eq:refined-dcp-yn}
\der y_{st}
=
\si(y_{s}) \, \der x_{st} + r_{st},
\quad\text{with}\quad
|r_{st}| \le   c_{3,x} \, 2^{-\ka_{\ep_{1}} q_{k}} |t-s|^{\ga},
\end{equation}
where we have set $\ka_{\ep_{1}}=\ka+2\ep_{1}\al$.
\end{proposition}

\begin{proof}
From \eqref{eq:decomp} we have 
\begin{equation}\label{eq:r_st}
|r_{st}| = |(D\si\cdot\si)(y_s)\mathbf{x}^{\mathbf{2}}_{st}+R_{st}| \le |(D\si\cdot\si)(y_s)||\mathbf{x}^{\mathbf{2}}_{st}|+|R_{st}| 
\end{equation}

\noindent
Under the constraints we have imposed on $s,t$, namely $s,t \in [\la_k, \la_{k+1})$ such that $|t-s|\le c_{0,x} 2^{-\al q_k}$, and recalling that we have set $\ga_1 = \ga + \ep_1$, we have 
\begin{multline}\label{eq:x2_bnd}
\sup_{s,t} \dfrac{|\mathbf{x}^{\mathbf{2}}_{st}|}{|t-s|^{\ga}} = \sup_{s,t} \dfrac{|\mathbf{x}^{\mathbf{2}}_{st}|}{|t-s|^{2\ga+2\ep_1}} {|t-s|^{\ga + 2 \ep_1}} \le \sup_{s,t} \dfrac{|\mathbf{x}^{\mathbf{2}}_{st}|}{|t-s|^{2\ga_1}} \sup_{s,t} {|t-s|^{\ga + 2\ep_1}}\\
 \le \cn\left[\mathbf{x^2};\cac_2^{2\ga_1}\right] (c_{0,x} 2^{-\al q_k})^{\ga + 2\ep_1}.
\end{multline}
where we have used $\sup_{s,t}$ to stand for supremum over the set $\{ (s,t): s,t \in [\la_k, \la_{k+1}) \text{~and~} |t-s|\le c_{0,x} 2^{-\al q_k}\}$.

Note that under Hypothesis~\ref{hyp:reg-x-gamma-gamma1}, the quantity ${\|\mathbf{x}\|}_{\ga_1}$ is finite and hence \eqref{eq:x2_bnd} can be read as:
\begin{equation}\label{eq:bnd:1}
\sup_{s,t} \dfrac{|\mathbf{x}^{\mathbf{2}}_{st}|}{|t-s|^{\ga}} \le {\|\mathbf{x}\|}_{\ga_1} c_{0,x}^{\ga + 2\ep_1} 2^{-\al(\ga + 2\ep_1)q_k}.
\end{equation}

Moreover, owing to \eqref{eq:R_3ga_norm} applied to $\ga := \ga + \ep_{1}$, and $\ka$ as in Hypothesis~\ref{hyp:sigma+}, we get
\begin{multline}\label{eq:prop_R_bnd}
\sup_{s,t} \dfrac{|R_{st}|}{|t-s|^{\ga}} = \sup_{s,t} \dfrac{|R_{st}|}{|t-s|^{3(\ga + \ep_1)}} {|t-s|^{2\ga + 3\ep_1}} \le \sup_{s,t} \dfrac{|R_{st}|}{|t-s|^{3\ga_1}} \sup_{s,t} |t-s|^{2\ga + 3\ep_1}\\
\leq \tilde{c}_{2,x} 2^{(2-3\ka)q_k} (c_{0,x} 2^{-\al q_k})^{2\ga + 3\ep_1}.
\end{multline}
Here we have used the notation $\tilde{c}_{2,x}$ to stand for the coefficient $c_{2,x}$ in \eqref{eq:R_3ga_norm}, with $\|\mathbf{x}\|_{\ga}$ replaced by $\|\mathbf{x}\|_{\ga_1}$. Thus we have
\begin{equation}\label{eq:bnd:2}
\sup_{s,t} \dfrac{|R_{st}|}{|t-s|^{\ga}} \leq \tilde{c}_{2,x} c_{0,x}^{2\ga + 3\ep_1} 2^{-(\al(2\ga+3\ep_1)+3\ka-2)q_k}
\end{equation}

Now incorporating \eqref{eq:bnd:1} and \eqref{eq:bnd:2} in \eqref{eq:r_st}, and recalling that $\alpha = \frac{1-\ka}{\ga}$, we easily get:
\begin{align*}
\sup_{s,t} \dfrac{|r_{st}|}{|t-s|^{\ga}} &\le \cn_{2\ka-1, D\si\cdot\si}\left(\dfrac{b_2}{2^{q_k}}\right)^{2\ka-1} {\|\mathbf{x}\|}_{\ga_1} c_{0,x}^{\ga + 2 \ep_1} 2^{-\al(\ga + 2 \ep_1)q_k} + \tilde{c}_{2,x} c_{0,x}^{2\ga + 3\ep_1} 2^{-(\al(2\ga + 3\ep_1)+3\ka-2)q_k}\\
&= \cn_{2\ka-1, D\si\cdot\si} b_2^{2\ka-1} {\|\mathbf{x}\|}_{\ga_1} c_{0,x}^{\ga + 2 \ep_1} 2^{-(\ka + 2\ep_1\al)q_k} + \tilde{c}_{2,x} c_{0,x}^{2\ga + 3\ep_1} 2^{-(\ka + 3 \ep_1 \al)q_k}
\end{align*} 
Collecting terms and recalling that we have set $\ka_{\ep_1} = \ka + 2 \ep_1 \al$, we end up with:
\begin{equation*}
\sup_{s,t} \dfrac{|r_{st}|}{|t-s|^{\ga}} \le c_{3,x} 2^{-(\ka + 2 \ep_1 \al)q_k} = c_{3,x} 2^{-\ka_{\ep_1} q_k}, 
\end{equation*}
which is our claim \eqref{eq:refined-dcp-yn}.
\end{proof}

\noindent
Thanks to our previous efforts, we can now slightly enlarge the interval on which our improved regularity estimates hold true: 
\begin{corollary}\label{cor:a-priori-bnd-yn-larger-intv}
Let the assumptions of Proposition~\ref{prop:regularity-gain} prevail, and consider $0 < \ep_1 < 1-\ga$ as in  Hypothesis~\ref{hyp:reg-x-gamma-gamma1}. Then with $\al= \ga^{-1}(1-\ka)$, there exists $0< \ep_2 < \alpha$ and a constant $c_{4,x}$ such that for all $(s,t) \in \mathcal{S}_2\left([\la_k, \la_{k+1})\right)$ satisfying $|t-s|\le  c_{4,x} 2^{-(\al-\ep_{2}) q_{k}} $  we have
\begin{equation}\label{eq:a-priori-bnd-yn-larger-intv}
\lln  \der y_{st} \rrn \le
 c_{5,x} 2^{-q_{k}\ka_{\ep_{2}}^{-}} 
 |t-s|^{\ga}   ,
\quad\text{where}\quad
\ka_{\ep_{2}}^{-} = \ka- (1-\ga)\ep_{2}.
\end{equation}
 Moreover, under the same conditions on $(s,t)$, decomposition \eqref{eq:refined-dcp-yn} still holds true, with
\begin{equation}\label{eq:refined-dcp-yn-larger-intv}
|r_{st}| \le c_{6,x}   2^{-q_{k} \ka_{\ep_{1},\ep_{2}}}  |t-s|^{\ga},
\quad\text{where}\quad
\ka_{\ep_{1},\ep_{2}} = \ka+ 2\al\ep_{1}-\ga\ep_{2} -2\ep_1\ep_2.
\end{equation}

\end{corollary}

\begin{proof}
We split our computations in 2 steps.

\noindent
\emph{Step 1: Proof of \eqref{eq:a-priori-bnd-yn-larger-intv}.}
Start from inequality \eqref{eq:y_ga_norm}, which is valid for $|t-s|\le  c_{0,x} 2^{-\al q_{k}}$. Now let $m\in\N$ and consider $s,t\in [\la_{k}, \la_{k+1})$ such that $c_{0,x}  (m-1) 2^{-\al q_{k}} <|t-s|\le  c_{0,x}  m 2^{-\al q_{k}}$. We partition the interval $[s,t]$ by setting $t_{j}=s+c_{0,x} j 2^{-\al q_{k}}$ for $j=0,\ldots,m-1$ and $t_{m}=t$. Then we simply write
\begin{equation*}
|\der y_{st}|
\le
\sum_{j=0}^{m-1} |\der y_{t_{j}t_{j+1}}|
\le
c_{1,x}  2^{-q_{k}\ka} \sum_{j=0}^{m-1} \lp t_{j+1} - t_{j} \rp^{\ga}
\le
c_{1,x} 2^{-q_{k}\ka} m^{1-\ga} |t-s|^{\ga},
\end{equation*}
where the last inequality stems from  the fact that $t_{j+1} - t_{j} \le (t-s)/m$. Now  the upper bound \eqref{eq:a-priori-bnd-yn-larger-intv} is easily deduced by applying the above inequality to a generic $m \leq [2^{\ep_{2} q_k}]+1$, where $ 0 < \ep_{2} < \frac{\ka}{1-\ga}$. This ensures $\ka_{\ep_{2}}^{-} = \ka- (1-\ga)\ep_{2} > 0$.

\noindent
\emph{Step 2: Proof of \eqref{eq:refined-dcp-yn-larger-intv}.} We proceed as in the proof of Proposition~\ref{prop:regularity-gain}, but now with a relaxed constraint on $(s,t)$, namely $|t-s| \leq c_{4,x} 2^{-(\al-\ep_{2}) q_{k}}$ where $\ep_2 > 0$ satisfies:
\begin{equation}\label{eq:ep_2_cond}
\ep_2 < \min \left(\frac{\ka}{1-\ga}, \frac{\ep_1 \alpha}{\ga + \ep_1} \right).
\end{equation}
The equivalent of relation \eqref{eq:prop_R_bnd} is thus
\begin{multline}\label{eq:cor_R_bnd}
\sup_{s,t} \dfrac{|R_{st}|}{|t-s|^{\ga}} = \sup_{s,t} \dfrac{|R_{st}|}{|t-s|^{3(\ga + \ep_1)}} {|t-s|^{2\ga + 3\ep_1}} \le \sup_{s,t} \dfrac{|R_{s,t}|}{|t-s|^{3\ga_1}} \sup_{s,t} |t-s|^{2\ga + 3\ep_1}\\
 \leq \tilde{c}_{2,x} 2^{(2-3\ka)q_k} (c_{4,x} 2^{-(\al-\ep_2) q_k})^{2\ga + 3\ep_1}
\end{multline}
As in Proposition~\ref{prop:regularity-gain} we have used the notation $\tilde{c}_{2,x}$ to stand for the coefficient $c_{2,x}$ with $\|\mathbf{x}\|_{\ga}$ replaced by $\|\mathbf{x}\|_{\ga_1}$ and $\sup_{s,t}$ to stand for supremum over the set $\lbrace (s,t): s,t \in [\la_k, \la_{k+1}) \text{~and~} |t-s|\le c_{4,x} 2^{-(\al-\ep_{2}) q_{k}}\rbrace$. Collecting the exponents in \eqref{eq:cor_R_bnd} we thus end up with:
\begin{equation}\label{eq:bnd__1}
\sup_{s,t} \dfrac{|R_{st}|}{|t-s|^{\ga}} \leq \tilde{c}_{2,x} c_{4,x} 2^{-(\ka + 3\ep_1\al - 2\ep_2\ga -3\ep_1\ep_2)q_k}.
\end{equation}
Similarly to \eqref{eq:x2_bnd}, we also get: 
\begin{multline}
\sup_{s,t} \dfrac{|\mathbf{x}^{\mathbf{2}}_{st}|}{|t-s|^{\ga}} = \sup_{s,t} \dfrac{|\mathbf{x}^{\mathbf{2}}_{st}|}{|t-s|^{2\ga+2\ep_1}} {|t-s|^{\ga + 2 \ep_1}} \le \sup_{s,t} \dfrac{|\mathbf{x}^{\mathbf{2}}_{st}|}{|t-s|^{2\ga_1}} \sup_{s,t} {|t-s|^{\ga + 2\ep_1}}\\
\le {\|\mathbf{x}\|}_{\ga_1} (c_{4,x} 2^{-(\al-\ep_2) q_k})^{\ga + 2\ep_1}.
\end{multline}
Consequently, owing to Hypothesis~\ref{hyp:si_der}, we get the following relation:
\begin{multline}\label{eq:bnd__2}
|(D\si\cdot\si)(y_s)\mathbf{x}^{\mathbf{2}}_{st}| \le \cn_{2\ka-1, D\si\cdot\si} \left(\dfrac{b_2}{2^{q_k}}\right)^{2\ka-1} \|\mathbf{x}\|_{\ga_1} c_{4,x}^{\ga+2\ep_1} 2^{-(\al - \ep_2)(\ga + 2\ep_1)q_k}\\ 
= \cn_{2\ka-1, D\si \cdot \si} b_2^{2\ka-1} \|\mathbf{x}\|_{\ga_1} c_{4,x}^{\ga + 2\ep_1} 2^{-(\ka +2\ep_1\al- \ep_2\ga -2\ep_1\ep_2)q_k}.
\end{multline}
Notice that under the conditions on $\ep_2$ in \eqref{eq:ep_2_cond}, we have $\ka + 2\ep_1 \al - \ep_2\ga - 2 \ep_1\ep_2 < \ka + 3 \ep_1 \al - 2\ep_2 \ga - 3\ep_1\ep_2$. Therefore incorporating \eqref{eq:bnd__1} and \eqref{eq:bnd__2} we have:
\begin{equation*}
|r_{st}| \leq |(D\si \cdot \si)(y_s)\mathbf{x}^{\mathbf{2}}_{st}|+|R_{st}| \lesssim 2^{-q_k \ka_{\ep_1, \ep_2}} |t-s|^{\ga}
\end{equation*}
which is our claim \eqref{eq:refined-dcp-yn-larger-intv}.
\end{proof}

\subsection{Estimates for stopping times}
Thanks to the previous estimates on improved regularity for the solution $y$ to  equation  \eqref{eq:sde-power}, we will now get a sharp control on the difference $\la_{k+1}-\la_{k}$. Otherwise stated we shall control the speed at which $y$ might converge to 0, which is the key step in order to control the global H\"older continuity of $y$. This section is similar to what has been done in \cite{YSDE}, and proofs are included for sake of completeness.
We start with a lower bound on the difference $\la_{k+1}-\la_{k}$.

\begin{proposition}\label{prop:upper-bound-diff-sigma-k}
Assume $\sigma$ and $(D\si \cdot \si)$ follows Hypothesis~\ref{hyp:sigma+}. Also assume Hypothesis~\ref{hyp:si_der} holds. Then the sequence of stopping times $\{\la_{k},\, k\ge 1\}$ defined by \eqref{eq:def-sigma-tau-k} satisfies
\begin{equation}\label{eq:low-bnd-increment-sigma-k}
\la_{k+1}-\la_{k} \ge  c_{5,x}  \, 2^{-\al q_{k}},
\end{equation}
where we recall that $\al=(1-\ka)/\ga$.
\end{proposition}

\begin{proof}
 We show that the difference $\tau_{k}-\la_{k}$ satisfies a lower bound of the form
\begin{equation}\label{eq:low-bnd-increment-tau-k-sigma-k}
\tau_{k}-\la_{k} \ge  c_{6,x} \, 2^{-\al q_k}.
\end{equation}
There exists a similar bound for $\la_{k+1}-\tau_{k}$, and consequently we get our claim~\eqref{eq:low-bnd-increment-sigma-k}. 

To arrive at inequality \eqref{eq:low-bnd-increment-tau-k-sigma-k} we observe that in order to leave the interval $[\la_{k},\tau_{k})$, an increment of size at least $2^{-(q_k+1)}$ must occur.  This is because at $\lambda_k$ the solution lies at the mid point of  $I_{q_k}$, an interval of size $2^{-q_k}$.
Thus, if $|\der y_{st}|\ge 2^{-(q_{k}+1)}$ and $|t-s|\le c_{0,x}  2^{-\al q_k} $, relation~\eqref{eq:y_ga_norm} provides us with: 
\begin{equation}\label{eq:lower-bnd-t-s}
c_{1,x}  \frac{|t-s|^{\ga}}{2^{\ka q_k}} \ge \frac{1}{2^{q_k+1}},
\end{equation}
which implies  
\begin{equation*}
 |t-s | \ge  \left(2 c_{1,x} \right)^{-\frac 1\ga} 2^{-\frac{(1-\ka)q_k}{\ga}}
 = \left(2c_{1,x} \right)^{-\frac 1\ga} 2^{-\al q_k}
 .
\end{equation*}
This completes the proof.
\end{proof}

In order to sharpen Proposition~\ref{prop:upper-bound-diff-sigma-k}, we introduce a roughness hypothesis on $x$, again as in \cite{YSDE}. This assumption is satisfied when $x$ is a fractional Brownian motion.

\begin{hypothesis}\label{hyp:roughness}
We assume that for $\hep$ arbitrarily small there exists  
a constant $c>0$ such that for every $s$ in $\left[  0,T\right]  $, every
$\epsilon$ in $(0,T/2]$, and every $\phi$ in $\mathbb{R}^{d}$ with $\left\vert
\phi\right\vert =1$, there exists $t$ in $\left[  0,T\right]  $ such that
$\epsilon/2<\left\vert t-s\right\vert <\epsilon$ and
\[
\left\vert \left\langle \phi,\der x_{st}\right\rangle \right\vert >c \, \epsilon^{\ga+\hep}.
\]
The largest such constant is called the modulus of $(\ga+\hep)$-H\"{o}lder
roughness of $x$, and is denoted by $L_{\ga,\hep}\left(  x\right)$.
\end{hypothesis}

Under this hypothesis, we are also able to upper bound the difference $\la_{k+1}-\la_{k}$.

\begin{proposition}\label{prop:bound-diff-sigma-k-2}
Assume $\sigma$ and $(D\si \cdot \si)$ follows Hypothesis~\ref{hyp:sigma+}. Also assume Hypothesis~\ref{hyp:si_der} holds and $\si(\xi) \gtrsim {|\xi|}^{\ka}$. Then for all $\ep_{2} < \frac{\al \ep_1}{\ga + \ep_1} \wedge \frac{\ka}{1-\ga}$ and $q_{k}$ large enough (that is for $k$ large enough, since $\lim_{k\to\infty}q_k= \infty$ under Assumption (B) of Proposition~\ref{prop:th1}), the sequence of stopping times $\{\la_{k},\, k\ge 1\}$ defined by \eqref{eq:def-sigma-tau-k} satisfies
\begin{equation}\label{eq:upp-bnd-increment-sigma-k}
\la_{k+1}-\la_{k}\le  c_{x,\ep_{2}}  2^{- q_{k} (\al-\ep_{2})},
\end{equation}
where we recall that $\al=(1-\ka)/\ga$. Furthermore, inequality \eqref{eq:a-priori-bnd-yn-larger-intv} can be extended as follows:  there exists a constant $c_{x}$ such that for $s,t\in [\la_{k}, \la_{k+1})$ we have
\begin{equation}\label{eq:a-priori-bnd-yn-whole-interval-lambda}
\lln  \der y_{st} \rrn \le
 c_{x} 2^{-\ka_{\ep_{2}}^{-}q_{k}}  |t-s|^{\ga}   ,
\end{equation}
\end{proposition}

\begin{proof}
We prove by contradiction. Assume the contrary, that is, \eqref{eq:upp-bnd-increment-sigma-k} does not hold. This implies that for some $\ep_{2} < \frac{\al \ep_1}{\ga + \ep_1} \wedge \frac{\ka}{1-\ga}$  
\begin{equation}\label{a1}
\la_{k+1}-\la_{k}\ge  C  2^{- q_{k} (\al-\ep_{2})}
\end{equation}
holds for infinitely many values of $k$, for any constant $C$. Consequently
\begin{equation}\label{a2}
\la_{k+1}-\la_{k} \ge  C \, 2^{-q_k(1-\ka)/(\ga+\hep)},
\end{equation}
for an $\hep$ small enough so that  $(1-\ka)/(\ga+\hep)\ge \al-\ep_{2}$. We now show that there exists $s,t \in [\la_k,\la_{k+1}]$ such that $|\delta y_{st}| > |J_{q_k}|$ providing us with our contradiction. Here $|J_{q_k}|$ denotes the size of the interval $J_{q_k}$.

To achieve this we now use Hypothesis~\ref{hyp:roughness}. Taking into account we are in the one-dimensional case let us choose
\[
\ep:=  \frac{c_{1} \, 2^{-\frac{q_k(1-\ka)}{\ga+\hep}}}{\lc L_{\ga,\hep}(x) \rc^{\frac{1}{\ga+\hep}}}  \le  C \, 2^{-\frac{q_k(1-\ka)}{\ga+\hep}},
\]
where the inequality is true for a fixed constant $c_1$ and a large enough constant $C$. Due to \eqref{a1} and Hypothesis~\ref{hyp:roughness} there now exist $s,t\in [\la_{k},\la_{k+1}]$ such that
\begin{equation}\label{eq:low-bnd-inc-x}
\frac \ep 2 \le |t-s| 
\le  \ep,
\quad\text{and}\quad
|\der x_{st}| \ge  \, c_{1}^{\ga+\hep} \, 2^{-q_k(1-\ka)}.
\end{equation}
Moreover, due to our assumptions on $\si$ and because $y_s \ge b_1 2^{-q_k} \ge 2^{ -q_k-2}$, we have $|\si(y_{s})|\ge c2^{-q_k\ka}$ for $s\in[\la_{k},\la_{k+1}]$.
 Consequently, for $s,t$ as in \eqref{eq:low-bnd-inc-x}
\begin{equation*}
|\si(y_{s}) \der x_{st}| \ge c c_{1}^{\ga+\hep} \, 2^{-q_k}.
\end{equation*}
For fixed $\ep$, $c_1$ can be chosen arbitrarily large (by increasing $k$ or decreasing $\hat{\ep}$) such that $cc_1^{\ga + \hat{\ep}} \geq 6$. We thus have 
\begin{equation*}
|\si(y_{s}) \der x_{st}| \ge   6\cdot  2^{-q_k} = 2 |J_{q_k}|.
\end{equation*}
In particular the size of this increment is larger than twice the size of $J_{q_k}$ (see relation \eqref{eq:def-sigma-tau-k}).

Recall, $\hep$ is small enough so that  $(1-\ka)/(\ga+\hep)\ge \al-\ep_{2}$, so that from the bound on $|t-s|$ in \eqref{eq:low-bnd-inc-x} we have $|t-s|\le c_{7,x}  2^{-q_k (\al-\ep_{2})}$.
With $s,t$ as in relation \eqref{eq:low-bnd-inc-x} we use the fact that ${\delta y}_{st} = \si(y_s) {\delta x}_{st} + r_{st}$ and the bound~\eqref{eq:refined-dcp-yn-larger-intv} to get
\begin{equation*}
|\der y_{st}| \gtrsim A_{st}^{1} - A_{st}^{2},
\quad\text{with}\quad
A_{st}^{1} = 6 \cdot 2^{-q_k},
\quad
A_{st}^{2} \le c_{6,x} 2^{-q_k \ka_{\ep_{1},\ep_{2}}}   |t-s|^{\ga} \le c_{9,x} 2^{-q_k \mu_{\ep_{2}}},
\end{equation*}
where we recall that $\ka_{\ep_{1},\ep_{2}} = \ka+ 2\al\ep_{1}-\ga\ep_{2} -2\ep_1\ep_2$ to obtain 
\begin{equation*}
\mu_{\ep_{2}}= \ka_{\ep_{1},\ep_{2}} +(\al-\ep_{2})\ga
= 1+2\al\ep_{1} -2(\ga+\ep_{1})\ep_{2}.
\end{equation*}
Compared to $2^{-q_k}$, $A_{st}^{2}$ can be made negligible for large enough $q_k$ by making sure that $\mu_{\ep_2} > 1$. One can ensure $\mu_{\ep_{2}}>1$ by choosing $\ep_{1}$ large enough and $\ep_{2}$ small enough.
As a consequence $|\der y_{st}| \gtrsim A_{st}^1 - A_{st}^2$, where $A_{st}^{1}$  is larger than twice $|J_{q_k}| =3\cdot 2^{-q_k}$ and $A_{st}^{2}$ is negligible compared to $A_{st}^{1}$ as $q_k$ gets large. That is, $|\der y_{st}| > |J_{q_k}|$ for $k$ large enough. We now have our contradiction and this proves \eqref{eq:upp-bnd-increment-sigma-k}. 
\end{proof}

\subsection{H\"older continuity}

Eventually the control of the stopping times $\la_{k}$ leads to the main result of this section, that is the existence of a $\cac^{\ga}$ solution to equation \eqref{eq:sde-power}. The crucial step in this direction is detailed in the proposition below. It is achieved under the additional assumption $\ga+\ka>1$, and yields directly the proof of Theorem \ref{thm:d-dim-power}.

\begin{proposition} \label{th2}
Suppose that our noise $x$ satisfies Hypotheses~\ref{hyp:reg-x-gamma-gamma1} and~\ref{hyp:roughness}. Assume $\sigma$ and $(D\si \cdot \si)$ follows Hypothesis~\ref{hyp:sigma+} and Hypothesis~\ref{hyp:si_der} holds as well. Also assume $\si(\xi) \gtrsim {|\xi|}^{\ka}$ and that $\ga+\ka>1$.
Then, the function $y$ given in Proposition~\ref{prop:th1} belongs to $\mathcal{C}^\gamma([0,T]; \R^m)$.
\end{proposition}
\begin{proof}
We start with the assumption that $y$ satisfies condition (B) in Proposition~\ref{prop:th1}.
 We first consider $s=\la_{k}$ and $t=\la_{l}$ with $k<l$ and decompose the increments $|\delta y_{st}|$ as:
\begin{equation*}
\lln \delta y_{st}  \rrn \leq \sum_{j=k}^{l-1} \lln \delta y_{\la_{j} \la_{j+1}} \rrn.
\end{equation*}
Due to Proposition~\ref{prop:bound-diff-sigma-k-2} we have $\la_{k+1}-\la_{k}\le  c_{x,\ep_{2}}  2^{- q_{k} (\al-\ep_{2})}$ for a large enough $k$. An application of Corollary~\ref{cor:a-priori-bnd-yn-larger-intv} yields
\begin{equation}\label{eq:global-increments-yn-1}
\lln \delta y_{st}  \rrn
\le
\sum_{j=k}^{l-1} \lln \delta y_{\la_{j} \la_{j+1}}  \rrn
\le
c_{5,x}  \sum_{j=k}^{l-1}  2^{-q_{j}\ka_{\ep_{2}}^{-}} |\la_{j+1} - \la_{j}|^{\ga}.
\end{equation}
Rewriting inequality \eqref{eq:low-bnd-increment-sigma-k}, 
\begin{equation*}
2^{-\frac{q_{j} (1-\ka)}{\ga}} \leq c_{7,x}^{-1} \lp  \la_{j+1} - \la_{j} \rp
\end{equation*}
which implies
\begin{equation*}
2^{-q_{j} \ka_{\ep_{2}}^{-}} \le (c_{7,x})^{-\frac{\ga \ka_{\ep_{2}}^{-}}{1-\ka}} \lp  \la_{j+1} - \la_{j} \rp^{\frac{\ga \ka_{\ep_{2}}^{-}}{1-\ka}}.
\end{equation*}
Using this inequality in \eqref{eq:global-increments-yn-1} and defining $c_{8,x}=c_{5,x} (c_{7,x})^{-\frac{\ga \ka_{\ep_2}^{-}}{1-\ka}}$, we get:
\begin{equation*}
\lln \delta y_{st}  \rrn
\le
c_{8,x} 
\sum_{j=k}^{l-1}   |\la_{j+1} - \la_{j}|^{\tilde{\mu}_{\ep_{2}}},
\quad\text{where}\quad
\tilde{\mu}_{\ep_{2}}= \ga\lp 1+  \frac{\ka_{\ep_{2}}^{-}}{1-\ka} \rp.
\end{equation*}
Recall $\ka_{\ep_{2}}^{-} = \ka - (1-\ga)\ep_{2}$, which can be made arbitrarily close to $\ka$. Hence under the assumption $\ga + \ka > 1$, $\tilde{\mu}_{\ep_{2}}$ is of the form $\tilde{\mu}_{\ep_{2}}=1+\ep_{3}$. We thus obtain
\begin{equation*}
\lln \delta y_{st}  \rrn
\le
c_{8,x} 
\sum_{j=k}^{l-1}   |\la_{j+1} - \la_{j}|^{1+\ep_{3}}
\le
c_{8,x} |\la_{l} - \la_{k}|^{1+\ep_{3}}
\le 
c_{8,x}  \, \tau^{1+\ep_{3}-\ga} |t-s|^{\ga},
\end{equation*}
where we recall $s=\la_k$ and $t=\la_l$. Having proved our claim for this special case, the general case for $s<\la_k \leq \la_l < t$ is obtained by the following decomposition
\begin{equation*}
\delta y_{st} = \delta y_{s \la_{k}}+ \delta y_{\la_{k} \la_{l}}  + \delta y_{\la_{l}t}.
\end{equation*}
Finally, we make use of \eqref{eq:a-priori-bnd-yn-whole-interval-lambda} in order to bound $\delta y_{s \la_{k}}$ and $\delta y_{\la_{l}t}$.
\end{proof}

\end{document}